\documentclass[10pt]{amsart}

\usepackage{amssymb}
\usepackage{graphicx}
\usepackage[cmtip,all]{xy}

\usepackage{amscd}

\newtheorem{theorem}{Theorem}[section]
\newtheorem{lemma}[theorem]{Lemma}

\newtheorem{prop}[theorem]{Proposition}
\newtheorem{cor}[theorem]{Corollary}

\theoremstyle{definition}
\newtheorem{definition}[theorem]{Definition}

\theoremstyle{remark}
\newtheorem{remark}[theorem]{Remark}
\numberwithin{equation}{section}

\newcommand \st[2]{{#1}^*({#2})}

\begin{document}

\title{Automorphisms of order three on numerical Godeaux surfaces}

\author{Eleonora Palmieri}
\address{Dipartimento di Matematica, Universit\`a degli Studi ``Roma Tre''\\ 
Largo San Leonardo Murialdo 1, I-00146, Roma\\ 
Italy\\
Fax: +39 06 5488 8072}
\email{palmieri@mat.uniroma3.it}

\date{}

\begin{abstract}
We prove that a numerical Godeaux surface cannot have an automorphism of order three.
\end{abstract}

\maketitle

\section{Introduction}

This paper is devoted to a basic open problem about  surfaces: the classification of surfaces of general type and their automorphisms.
We will  work over the complex numbers. 
Complex surfaces have been classified by Enriques and Kodaira in terms of their Kodaira dimension $\kappa$.

While surfaces with $\kappa\le 1$ are quite well-known, we have much less information about surfaces of general type, i.e. those for which $\kappa=2$. Their complete classification is still an open problem even though there are important contributions from many mathematicians (for a general reference see \cite{BCP}).

We know that minimal surfaces of general type are subdivided into classes according to the value of three main invariants: the self-intersection of the canonical divisor $K_S^2$, the holomorphic Euler characteristic $\chi(S,\mathcal O_S)$ and the geometric genus $p_g(S):=h^0(S,\mathcal O_S(K_S))=h^2(S,\mathcal O_S)$. 
Here we are mainly interested in those surfaces with the lowest invariants:
\begin{definition}\label{d.Go}
A {\bf numerical Godeaux surface} is a minimal complex surface of general type $S$ with $p_g(S)=0, K_S^2=1,\chi(\mathcal O_S)=1$.
\end{definition}

The first example of such a surface can be found in \cite{G} and it is the quotient of a smooth quintic in $\mathbb P^3$ with a free $\mathbb Z/5\mathbb Z$ action. This example turns out to have non-trivial torsion, and in fact it has $\mathbb Z/5\mathbb Z$ as a torsion group.

Much information about the torsion group of numerical Godeaux surfaces can be obtained by the study of the base points of the tricanonical system $|3K_S|$. This is an important result by Miyaoka (see \cite{M}). 
It is known (see \cite{M, R}) that the moduli spaces of numerical Godeaux surfaces with torsion group $\mathbb Z/3\mathbb Z$,  $\mathbb Z/4\mathbb Z$ and  $\mathbb Z/5\mathbb Z$  are irreducible of dimension 8.

 As for every surface of general type $Aut(S)$ is a finite group (see also \cite{X1, X2, X}). It is still a quite difficult problem to determine the group $Aut(S)$.

The simplest case is that of a surfaces $S$ admitting an involution, i.e. an automorphism of order 2. For Godeaux surfaces in \cite{KL} Keum and Lee study the fixed locus of the involution under the hypothesis that the bicanonical system $|2K_S|$ of the surface has no fixed component.

In their work \cite{CCM} Calabri, Ciliberto and Mendes Lopes complete the above study by removing this hypothesis. Their result is the following:
 \begin{theorem}
A numerical Godeaux surface $S$ with an involution is birationally equivalent to one of the following:
\begin{enumerate}
\item a double plane of Campedelli type;
\item a double plane branched along a reduced curve which is the union of two distinct lines $r_1,r_1$ and a curve of degree 12 with the following singularities:
\begin{enumerate}
\item[$\bullet$] the point $q_0=r_1\cap r_2$ of multiplicity 4;
\item[$\bullet$] a point $q_i\in r_i$, $i=1,2$ of type $[4,4]$, where the tangent line is $r_i$;
\item[$\bullet$] further three points $q_3,q_4,q_5$ of multiplicity 4 and a point $q_6$ of type $[3,3]$, such that there is no conic through $q_1,\dots,q_6$;
\end{enumerate}
\item a double cover of an Enriques surface branched along a curve of arithmetic genus 2.
\end{enumerate}
In case 3 the torsion group of $S$ is ${\rm Tors}(S)=\mathbb Z/4\mathbb Z$, whilst in case 2 is either $\mathbb Z/2\mathbb Z$ or $\mathbb Z/4\mathbb Z$.
 \end{theorem}

We recall that a double plane of Campedelli type is a double plane branched along a curve of degree 10 with a 4-tuple point and 5 points of type $[3,3]$, not lying on a conic.
An example of such a double plane can be found in \cite{S}.

 We want to extend the method used in \cite{CCM} in order to classify such numerical Godeaux surfaces $S$ having an automorphism $\sigma$ of order three.
Our main result is 
\begin{theorem}\label{t.final}
A numerical Godeaux surface $S$ cannot have an automorphism of order 3.
\end{theorem}

It is possible to construct (see also \cite{Ca, T}) a minimal smooth resolution of the cover $p:S\longrightarrow \Sigma=S/\sigma$, i.e. a commutative diagram
\begin{equation}\label{e.CD}
\begin{CD}
X    @>\varepsilon>>  S\\
@V\pi VV    @VVpV\\
Y    @>\eta>> \Sigma
\end{CD}
\end{equation}
where $X$ and $Y$ are smooth surfaces and $\pi:X\longrightarrow Y$ is the triple cover induced by $\sigma$. The main idea is then to apply the theory of abelian covers following \cite{P}.

We start our analysis, using Hurwitz formula and the topological Euler characteristic $e$ to estimate the number of isolated fixed points of the action of $\sigma$ on $S$. Such points can be mapped either to ordinary triple points or to double points of type $A_2$. We determine some basic properties of the invariant part $\Lambda$ of the tricanonical system $|3K_S|$, which can be either a pencil or a net and it is mapped to a system $|N|$ over the quotient surface $Y$. Moreover we study the adjoint systems to $|N|$ with the help of \cite[lemma 2.2]{CCM2}. All the relevant  numerical properties are collected in proposition \ref{p.comp}. We also have a subdivision in three major cases (see the list of page \pageref{list}) according to the intersection number $R_0K_S$ and $h_2$, where $R_0$ is the divisorial part of the ramification locus of $\sigma$ while $h_2$ is the number of isolated fixed points of $\sigma$ mapped to $A_2$-singularities.

A numerical analysis of these three cases is worked out in sections \ref{s.i}, \ref{s.ii}, \ref{s.iii} where using some properties of nef divisors and  fibrations it is shown (see theorems \ref{t.i} and \ref{t.ii}) that the first two cases cannot occur.
In the third case the system $|N|$ on $Y$ (and also $\Lambda$ on $S$) is a pencil and its movable part induces a fibration over $Y$. An analysis of the singular fibres determines the  possibilities listed in theorem \ref{t.iii}. It is quite easy to see, although it is a very important information, that $Y$ is a smooth rational surface (proposition \ref{p.rat}).

Sections \ref{s.more} and \ref{s.other} are devoted to a deeper study of the adjoint systems to the pencil $|N|$ and to exclude some of the cases coming from theorem \ref{t.iii}. We also divide the remaining group of cases between Del Pezzo cases and ruled cases (see definitions \ref{d.rul} and \ref{d.DP}), since either $Y$ is a blow-up of $\mathbb P^2$ at a certain number of points, or $Y$ has a rational pencil with self-intersection 0. Moreover we show that the divisorial part $R_0$ of the ramification locus of the order three automorphism $\sigma$ on the numerical Godeaux surface $S$ is either 0 or it has only one irreducible component. 

Last sections deal with a more geometric study. We first analyze the ruled cases.
We show that $Y$ after contraction of suitable curves can be mapped onto $\mathbb F_0,\mathbb F_1$ or $\mathbb F_2$ and that, by blowing up a point and contracting again, we can always reduce to $\mathbb F_1$. Then we can actually see, birationally speaking, our surface $S$ as triple plane.

A computation of the movable part $|A'|$ of the pencil $|N|$ on $Y$ allows us to show that ruled cases cannot actually occur. 

Finally we study the Del Pezzo cases where the rational surface $Y$ is mapped to the projective plane blown-up at seven, eight or thirteen points.
The computation of the exceptional curves coming from the blow-up of the isolated fixed points on $S$ tells us that also Del Pezzo cases do not occur. 

One might now ask whether there are numerical Godeaux surfaces with automorphisms of order $p>3$ and, if so, might want to classify them. As we have seen, this is not an easy problem in general. However we notice that Stagnaro's  construction (see \cite{S}) gives us an example of a numerical Godeaux surface $S$ with an order 5 automorphism. In fact in this case the surface $S$ is birationally equivalent to a double plane
\begin{equation}\label{e.double}
z^2=f_{10}(x,y)
\end{equation}
where $f_{10}(x,y)$ is an irreducible polynomial of degree 10 which is invariant under the plane transformation 
$(x,y)\longrightarrow (\lambda x,\lambda^2 y)$ where   $\lambda=e^{2\pi i/5}$.
One can easily show that 
\begin{equation*}
(x,y,z)\longrightarrow (\lambda x,\lambda^2 y,z)
\end{equation*}
is an automorphism of order 5 on \eqref{e.double} hence on the numerical Godeaux surface $S$.
Thus the non-existence of order 3 automorphisms on numerical Godeaux surfaces appears as a quite surprising result.

The results contained in this paper are part of the author's Ph.D. thesis \cite{Pa} which can be also found  at the following web address 

\begin{center}
 http://ricerca.mat.uniroma3.it/dottorato/Tesi/tesipalmieri.pdf.
\end{center}

\subsection*{Notation}
Throughout the paper linear equivalence of divisors is denoted by $\equiv$, whereas numerical equivalence is denoted by $\sim$. The intersection product of two divisors $A$ and $B$ on a surface  is denoted by $AB$.  The remaining notation is standard in algebraic geometry.
 
\subsection*{Acknowledgements}
I would like to heartily thank prof. Ciro Ciliberto, who introduced me to this problem, for his guidance and his constant encouragement during the preparation of \cite{Pa}.
I also wish to thank prof. Margarida Mendes Lopes and prof. Fabrizio Catanese for their careful reading and for many useful suggestions and remarks.

\section{Preliminary results}

Let us consider a numerical Godeaux surface $S$ (see definition \ref{d.Go}) with an order 3 automorphism $\sigma$ and let $p:S\longrightarrow \Sigma$ be the projection of $S$ to its quotient $\Sigma=S/<\sigma>$. 
Let also  $\pi:X\longrightarrow Y$ be the resolution of the cover $S\longrightarrow \Sigma$ with $X$ and $Y$ smooth as in \cite{Ca, T}.
So we have the commutative diagram \eqref{e.CD}.

Let us fix the notation: $R_0$ is the ramification divisor of $p$, $h_1$ is the number of isolated fixed points $p_i$ of $\sigma$ which descend to triple point singularities of $\Sigma$, whereas $h_2$ is the number of isolated fixed points $q_j$ of $\sigma$ which descend to double point singularities of $\Sigma$. 
We also set  $E=\sum_{i=1}^{h_1}E_i$ where $E_i$ is the exceptional curve corresponding to the point $p_i$. We will denote the reducible $(-1)$-curve which contracts to a point $q_j$  by $F_j+G_i+H_j$ where $F_j,H_j$ are $(-1)$-curves and $G_j$ is a $(-3)$-curve with $F_jG_j=H_jG_j=1$, $F_jH_j=0$. The sum of the curves $F_i$, $G_i$ and $H_i$ will be similarly denoted by $F$, $G$, $H$.  Let finally $B_0=\pi(\st\varepsilon{R_0})$ and $E'_i$, $F'_i$ etc. be  the images of $E_i$, $F_i$,... via $\pi$.

So we have $R=Ram(\pi)=\st\varepsilon{R_0}+E+F+H$ and, by Hurwitz formula,
\begin{equation}\label{e.Hur}
K_X=\st\pi{K_Y}+2R=\st\pi{K_Y}+2\st\varepsilon{R_0}+2E+2F+2H
\end{equation}
while since $X$ is a blow-up of $S$
\begin{equation}\label{e.Bl}
K_X=\st\varepsilon{K_S}+E+2F+G+2H
\end{equation}

\begin{lemma}
We have
\begin{equation}\label{e.rk}
\st\varepsilon{R_0} K_X=R_0K_S, \quad \st\varepsilon{R_0}\st\pi{K_Y}=B_0K_Y
\end{equation}
\begin{equation}\label{e.BKY}
B_0K_Y=R_0K_S-2R_0^2.
\end{equation}
\end{lemma}

\begin{proof}
Let us compute $\st\varepsilon{R_0} K_X$ using formulas \ref{e.Hur} and \ref{e.Bl}. We notice that, since $\st\pi {B_0}=3R_0$,
$\st\varepsilon{R_0} \st\pi{K_Y}=B_0K_Y$.
By \eqref{e.Hur} we obtain
\begin{equation*}
\st\varepsilon R_0 K_X=\st\varepsilon R_0(\st\pi{K_Y}+2\st\varepsilon R_0+2E+2F+2H)
=B_0 K_Y+2R_0^2
\end{equation*}
Instead, by \eqref{e.Bl} we find
\begin{equation*}
\st\varepsilon{R_0} K_X=\st\varepsilon{R_0}(\st\varepsilon{K_S}+E+2F+G+2H)=R_0K_S
\end{equation*}
The desired result follows.
\end{proof}

\begin{prop}
Let $S$, $\sigma$, $X$, $Y$ be as above. Then the number of isolated fixed points of $\sigma$ satisfies 
\begin{equation}\label{e.fixed}
h_1+2h_2=6+\frac{3R_0K_S-R_0^2}2.
\end{equation}
Moreover we have 
\begin{equation}\label{e.ky}
K_Y^2=\frac 1 3[K_S^2-(h_1+3h_2)+4R_0^2-4R_0K_S].
\end{equation}
\end{prop}

\begin{proof}
Computing the Euler numbers of $X$ and $Y$ we obtain
\begin{equation}\label{e.Eu}
e(X)=3e(Y)-2e(R).
\end{equation}
Now,
\begin{equation*}
-e(R)=-e(\st\varepsilon{R_0})-2(h_1+2h_2)=R_0^2+R_0K_S-2(h_1+2h_2)
\end{equation*}
\begin{equation*}
e(X)=12-K_X^2,\quad e(Y)=12-K_Y^2
\end{equation*}
so from \eqref{e.Eu}
\begin{equation}\label{e.eu2}
12-K_X^2=3(12-K_Y^2)+2(R_0^2+R_0K_S)-4(h_1+2h_2)
\end{equation}
Again from \eqref{e.Hur}, \eqref{e.rk} and \eqref{e.BKY}
\begin{equation}\label{e.KX}
K_X^2=(\st\pi{K_Y}+2\st\varepsilon R_0+2E+2F+2H)^2=3K_Y^2-4R_0^2+4R_0K_S
\end{equation}
hence
\begin{equation*}
K_Y^2=\frac 1 3[K_X^2+4R_0^2-4R_0K_S]=\frac 1 3[K_S ^2-(h_1+3h_2)+4R_0 ^2-4R_0K_S]
\end{equation*}

Putting all these together and substituting \eqref{e.KX} in \eqref{e.eu2} we obtain
\begin{equation*}
12-3K_Y^2+4R_0^2-4R_0K_S=36-3K_Y^2+2(R_0^2+R_0K_S)-4(h_1+2h_2)
\end{equation*}
from which we infer
\begin{equation*}
h_1+2h_2=6+\frac{3R_0K_S-R_0^2}2
\end{equation*}
as wanted.
\end{proof}

\begin{remark}
Using the above proposition we immediately have
\begin{equation}\label{e.KY2}
K_Y^2=\frac 1 3[K_S ^2-6-h_2+\frac 9 2R_0 ^2-\frac{11} 2R_0K_S].
\end{equation}
\end{remark}

We have 
\begin{eqnarray*}
&2K_X-R=2(\st\pi{K_Y}+2R)-R=\st\pi{2K_Y}+3R=\st\pi{2K_Y+B}\\
&3K_X=3(\st\pi{K_Y}+2R)=\st\pi{3K_Y+2B}
\end{eqnarray*}

and from the theory of abelian triple covers (see \cite{Mi, P}) 
\begin{equation}\label{e.split}
\pi_*\mathcal O_X=\mathcal O_Y\oplus\mathcal O_Y(-L_1)\oplus\mathcal O_Y(-L_2).
\end{equation}
Then
\begin{align}\label{eq.split1}
\pi_*(\mathcal O_X(2K_X-R))&=\mathcal O_Y(2K_Y+B)\otimes(\mathcal O_Y\oplus\mathcal O_Y(-L_1)\oplus\mathcal O_Y(-L_2))\\
&=\mathcal O_Y(2K_Y+B)\oplus \mathcal O_Y(2K_Y+L_2)\oplus \mathcal O_Y(2K_Y+L_1)\notag
\end{align}
\begin{align}\label{eq.split2}
&\pi_*(\mathcal O_X(3K_X))=\mathcal O_Y(3K_Y+2B)\otimes(\mathcal O_Y\oplus\mathcal O_Y(-L_1)\oplus\mathcal O_Y(-L_2))\\
&=\mathcal O_Y(3K_Y+2B)\oplus \mathcal O_Y(3K_Y+B+L_2) \oplus \mathcal O_Y(3K_Y+B+L_1)\notag
\end{align}

In particular 
\begin{equation*}
2=h^0(X,\mathcal O_X(2K_X))\ge h^0(X,\mathcal O_X(2K_X-R))\ge h^0(Y,\mathcal O_Y(2K_Y+B))\ge 0
\end{equation*}
\begin{equation*}
4=h^0(X,\mathcal O_X(3K_X))\ge h^0(Y,\mathcal O_Y(3K_Y+2B))\ge 0
\end{equation*}

\begin{remark}\label{r.4}
We note that  $ h^0(Y,\mathcal O_Y(3K_Y+2B))=4$ cannot occur, because if so, then each curve of the tricanonical system $|3K_X|$ would be invariant under the action of $\sigma$,  hence the tricanonical map $\phi_{|3K_X|}$ would be composed with $\sigma$: this is not possible since $\phi_{|3K_X|}$ is a birational map (see \cite{M}). 
\end{remark}

\begin{lemma}\label{le.N}
The divisor $N=3K_Y+2B_0+E'-3G'$ on $Y$ is nef and big with $N^2=3,NK_Y=1-2R_0K_S$.
\end{lemma}

\begin{proof}
We just  observe that $\st\pi N=\st\varepsilon{3K_S}$.
which is nef and big since $S$ is of general type.
\end{proof}

We now want to apply Kawamata-Viehweg theorem (see for example \cite{BPV}) to compute the dimensions  of $H^0(Y,\mathcal O_Y(3K_Y+2B))$ and $H^0(Y,\mathcal O_Y(2K_Y+B))$ as vector spaces. We obtain the following
\begin{prop}\label{p.rk}
In the above setting we have
\begin{enumerate}
\item[(a)] $h^0(Y,\mathcal O_Y(3K_Y+2B))=h^0(Y,\mathcal O_Y(N))=2+R_0K_S$

\item[(b)] $h^0(Y,\mathcal O_Y(2K_Y+B))=\frac 1 3(2h_2-2-R_0K_S)$
\end{enumerate}
Moreover, we have $0\le R_0K_S\le 1$ and it can be $R_0K_S=1$ if and only if $h^0(Y,\mathcal O_Y(2K_Y+B))=1$ and $h_2=3$.
\end{prop}

\begin{proof}
(a) We determine some curves in the fixed part of $|3K_Y+2B|$. 
We can write $|3K_Y+2B|=|N|+E'+2F'+2H'+3G'$. 
 So we have $h^0(Y,\mathcal O_Y(3K_Y+2B))=h^0(Y,\mathcal O_Y(N))$.
Moreover, since $\st\pi{N}=\st\varepsilon{3K_S}$, using the formula
\begin{equation*}
\pi_*(\mathcal O_X(\st\varepsilon{3K_S}))=\mathcal O_Y(N)\oplus \mathcal O_Y(N-L_1)\oplus \mathcal O_Y(N-L_2)
\end{equation*}
and the fact that $h^i(S,\mathcal O_S(3K_S))=0$ for all $i>0$, we find $h^i(Y,\mathcal O_Y(N))=0$ for all $i>0$. Then, using lemma \ref{le.N} one has
\begin{equation*}
0\le  h^0(Y,\mathcal O_Y(N))= \chi(Y,\mathcal O_Y(N))=1+\frac{N(N-K_Y)}2
=2+R_0K_S 
\end{equation*}

(b) Again we determine some curves in the fixed part of $|2K_Y+B|$. 
We have 
\begin{equation*}
h^0(Y,\mathcal O_Y(2K_Y+B))=h^0(Y,\mathcal O_Y(2K_Y+B_0+E'-G'))
\end{equation*}
But we can also write
\begin{equation*}
2K_Y+B_0+E'-G'=K_Y+(K_Y+B_0+E'-G')=K_Y+\frac 1 3 N+\frac 1 3 B_0+\frac 2 3 E'
\end{equation*}
and by Kawamata-Viehweg theorem 
\begin{equation*}
h^i(Y,\mathcal O_Y(2K_Y+B_0+E'-G'))=0 \quad \text{for all $i>0$.}
\end{equation*}
Then, as in (a), using \eqref{e.BKY}, \eqref{e.fixed}, \eqref{e.ky} we have
\begin{align*}
0&\le h^0(Y,\mathcal O_Y(2K_Y+B_0+E'-G'))=\chi(Y,\mathcal O_Y(2K_Y+B_0+E'-G'))\\
&=1+\frac{(2K_Y+B_0+E'-G')(K_Y+B_0+E'-G')}2\\
&=\frac 1 3[3+K_S^2-6+2h_2-R_0K_S]=\frac1 3(2h_2-2-R_0K_S)
\end{align*}
The last assertion follows  by remark \ref{r.4} and $0\le h^0(Y,\mathcal O_Y(2K_Y+B))\le 2$.
\end{proof}

So we are left with only three possible cases, according to the values of $R_0K_S$  and of $h_2$:

\begin{enumerate}
\item[(i)] $h^0(Y,\mathcal O_Y(N))=3$, $h^0(Y,\mathcal O_Y(2K_Y+B))=1$, $R_0K_S=1$, $h_2=3$ \label{list}
\item[(ii)] $h^0(Y,\mathcal O_Y(N))=2$, $h^0(Y,\mathcal O_Y(2K_Y+B))=2$, $R_0K_S=0$, $h_2=4$ 
\item[(iii)] $h^0(Y,\mathcal O_Y(N))=2$, $h^0(Y,\mathcal O_Y(2K_Y+B))=0$, $R_0K_S=0$, $h_2=1$ 
\end{enumerate}

\begin{lemma}\label{l.chili}
For any $1\le i\le 2, 0\le j\le 2$ we have $h^j(Y,\mathcal O_Y(-L_i))=0$.
In particular $L_i^2+L_iK_Y=-2$ for $i=1,2$.
\end{lemma}

\begin{proof}
It is immediate from  \eqref{e.split} since $X$ is birational to a numerical Godeaux surface and $Y$ is smooth. 
\end{proof}

\begin{prop}\label{p.h0li1}
Assume  case $(iii)$ above holds and $\ell=1$. Then $R_0$ is an irreducible $(-2)$-curve and $h_1=4+\ell=5$. Let $\omega=e^{\frac{2\pi i}3}$ be a primitive third root of unity and let $h_{11}$ and $h_{12}$ be the number of curves $E_i$ such that the eigenvalue of the action of $\mathbb Z/3\mathbb Z$ on $E_i$ is $\omega$ and $\omega^2$ respectively. Then if $\omega$ is the eigenvalue corresponding to $R_0$ then $h_{11}=2,h_{12}=3$. 
\end{prop}

\begin{proof}
Since case (iii) holds from \eqref{e.fixed} we infer $h_1=4+\ell=5$. 

We now write as $\bar E'_+$ and $\bar E'_-$ the sum of the curves $E'_i$ associated to the same eigenvalue $\omega$ and $\omega^2$ respectively. Since $h_1=5=h_{11}+h_{12}$ from the theory of abelian triple covers we have 
\begin{equation*}
3L_1\equiv B_0+E'_++2\bar E'_-+F'+2H'
\end{equation*}
and we find
\begin{equation*}
L_1K_Y=\frac 1 3(B_0+E'_++2\bar E'_-+F'+2H')K_Y=\frac{14-h_{11}}3+1
\end{equation*}
hence $h_{11}\equiv 2 \mod 3$ that forces $h_{11}=2,h_{12}=3$ or $h_{11}=5,h_{12}=0$.
Furthermore  
\begin{equation*}
L_1^2=\frac 1 9(B_0+E'_++2\bar E'_-+F'+2H')^2=-9+h_{11}
\end{equation*}

From lemma \ref{l.chili} we know that $L_1^2+L_1K_Y=-2$ hence
\begin{equation*}
-2=L_1^2+L_1K_Y=-9+h_{11}+\frac{14-h_{11}}3+1=\frac{2h_{11}-10}3
\end{equation*}
and $h_{11}=2$.
\end{proof}

\section{The invariant part of the tricanonical system}\label{s.inv}

Before going on, we want to better understand the properties of the curves in  $|N|$ (which is always non-empty). In particular, in lemma \ref{le.N} we have seen that $N^2=3$ and $NK_Y=1-2R_0K_S$ so that
\begin{equation*}
p_a(N)=1+\frac{N^2+NK_Y}2=1+\frac{3+1-2R_0K_S}2=3-R_0K_S.
\end{equation*}

\begin{lemma}\label{le.sub}
Let $S$ be a numerical Godeaux surface and  let $\Lambda$ be a linear subsystem of $|3K_S|$ with $\dim \Lambda\ge 1$ and $\Lambda=\mathcal A+\Phi$ where $\mathcal A$ is the movable part and $\Phi$ is the fixed part of $\Lambda$. Then the general member  $A\in\mathcal A$ is reduced and irreducible and one of the following conditions is satisfied:
\begin{enumerate}
\item[a)] $AK_S=2$, $\Phi K_S=1$, $A^2=0,2,4$ , $p_a(\Phi)\le 2$
\item[b)] $AK_S=3$, $\Phi K_S=0$ and either $A^2=1,3,5,7$, $p_a(\Phi)\le 0$ or $\Phi=0$.
\end{enumerate}
Moreover, if $A^2=4$ then $A\sim 2K_S$. 
\end{lemma}

\begin{proof}
We have $3=3K_S^2=\Lambda K_S=AK_S+\Phi K_S$.
Moreover, by Miyaoka \cite{M}, we know that $AK_S\ge 2$. This implies either $AK_S=2$, $\Phi K_S=1$ or $AK_S=3$, $\Phi K_S=0$. In the former case by the Index theorem (see  \cite{BPV})
\begin{equation*}
0\ge (A-2K_S)^2=A^2+4-8=A^2-4
\end{equation*}
and
\begin{equation*}
0\ge (\Phi-K_S)^2=\Phi^2+1-2=\Phi^2-1
\end{equation*}
which proves a).
A similar argument shows b).
To see the irreducibility of $A$ simply observe that if $A=A_1+A_2$ was reducible then $A_1K_S, A_2K_S\ge 2$ and $AK_S\ge 4$. Contradiction.
\end{proof}

\begin{prop}\label{pr.fix}
If the linear system $|N|$ has fixed part, then $|N|=|A'|+\Phi'$ with ${A'}^2=0,1,2$ and the general curve of $|A'|$ is smooth. Moreover $A'N=AK_S$,  $A'B_0=AR_0$.
\end{prop}

\begin{proof}
Since $\st\pi N=\st\varepsilon{3K_S}$ there is a linear subsystem $\Lambda$ of $|3K_S|$ such that $\st\varepsilon\Lambda=\st\pi{|N|}$ and $\dim \Lambda=h^0(Y,\mathcal O_Y(N))-1=1+R_0K_S$. Thus we can apply lemma \ref{le.sub} to $\Lambda$. 
Moreover the strict transform $\widetilde A$ of $A$ is the movable part of $\st\pi{|N|}$, so $\widetilde A=\st\pi {A'}$ where $|A'|$ is the movable part of $|N|$. Then
\begin{equation*}
9\ge \st\varepsilon A^2\ge \widetilde A^2=\st\pi{A'}^2=3{A'}^2.
\end{equation*}
This forces $\widetilde A^2$ to be 0, 3, 6 or 9. If $\widetilde A^2=9$ then $\widetilde A=A$ and the linear system $\Lambda$, hence $|N|$, has no fixed part. The last assertion is an easy computation.
\end{proof}

\noindent
We now focus our attention on the case $\dim \Lambda=1+R_0K_S=1$ or equivalently $R_0K_S=0$. Then $\mathcal A$ is a pencil and $A^2$ is the number of base points of $\mathcal A$. 

\begin{remark}\label{r.R0}
We note that, if $R_0K_S=0$, since $\Lambda=A+\Phi\equiv 3K_S$, for each irreducible component $R_{0i}$ of $R_0$, we have either $AR_{0i}=0$ or $R_{0i}\le \Phi$. Then
$AR_0=\sum_{i=1}^\ell AR_{0i}\le A\Phi$.
On the other hand
\begin{equation*}
9=\Lambda^2=A^2+2A\Phi+\Phi^2
\end{equation*}
and
\begin{equation*}
3\Phi K_S=\Phi\Lambda=A\Phi+\Phi^2.
\end{equation*}
Therefore
\begin{equation*}
0\le AR_0\le A\Phi=9-A^2-3\Phi K_S.
\end{equation*}
Moreover the $A\Phi$ intersection points between $A$ and $\Phi$ form an invariant set for the action of $\mathbb Z/3\mathbb Z$ on $S$.
\end{remark}

Let us write $\st\varepsilon A=\widetilde A+D$ with $D$ a sum of exceptional divisors with certain multiplicities.

\begin{remark}\label{r.L} 
Let us write $\st\varepsilon \Phi=\widetilde \Phi+D'$.
Then there exists a divisor $\Phi''$ on $Y$ such that $\st\pi{\Phi''}=\widetilde \Phi$ and
$\st\pi{\Phi'}=\widetilde \Phi+D+D'$.
This implies $(D+D')^2\equiv 0\mod 3$. Moreover,  the multiplicity of each curve $E_k$, $F$ or $H$  in $D+D'$ is a multiple of 3, since they appear in the branch locus of the cover $\pi: X\longrightarrow Y$ and $D+D'=\st\pi{\Phi'-\widetilde\Phi}$ is a pull-back of a divisor on $Y$.

We also remark that if $\Phi=0$ we have $\widetilde\Phi\equiv D'\equiv 0$ hence $\st\pi{\Phi'}\equiv D$.
\end{remark}

\begin{lemma}\label{l.drop1}
 For each simple base point of $A$ which is an isolated fixed point $q_j$ the self-intersection $\widetilde A^2$ of $\widetilde A$ drops exactly by 2. Moreover either $\st\varepsilon A=\widetilde A+2F_j+G_j+H_j$ or  $\st\varepsilon A=\widetilde A+F_j+G_j+2H_j$.
\end{lemma}

\begin{proof}
We simply blow up $q_j$ as shown in \cite{Ca} or \cite{T} and compute $\st\varepsilon A$. 
\end{proof}

Similarly one can show
\begin{lemma}\label{l.drop2}
 For each double base point of $A$ which is an isolated fixed point $q_j$ the self-intersection $\widetilde A^2$ of $\widetilde A$ drops at least by 5. In any case this can only happen when $A^2\ge 6$. Moreover, if $q_j$ is a node then $\st\varepsilon A=\widetilde A+3F_j+2G_j+3H_j$, if $q_j$ is a cusp then $\st\varepsilon A=\widetilde A+3F_j+2G_j+2H_j$. Finally if $q_j$ is neither a node nor a cusp then $A^2=9$ and $\st\varepsilon A=\widetilde A+4F_j+2G_j+2H_j$.
\end{lemma}

\begin{lemma}\label{l.drop3}
If the general $A\in\mathcal A$ has a triple point singularity at one of the isolated fixed point $q_j$ we have $A^2=9$ and $q_j$ is an ordinary triple point. Moreover $\st\varepsilon A=\widetilde A+3F_j+3G_j+3H_j$.
\end{lemma}

\begin{remark}\label{r.A9} 
From remark \ref{r.L} when $A^2=9$ (or equivalently $\Phi=0$) we have $D'=0$ and each component of $D$ different from $G$ has multiplicity $m\equiv 0 \mod 3$. In particular if we look at the multiplicities $\alpha_j$ of $A$ at the points $q_j$ we find, using lemmas \ref{l.drop1}, \ref{l.drop2} and \ref{l.drop3}, the following possibilities:
\begin{enumerate}
\item  $\alpha_j=0$ 
\item   $\alpha_j=2$ and  $q_j$ is a node   
\item   $\alpha_j=3$ and $q_j$ is an ordinary triple point.
\end{enumerate}

Moreover the multiplicity $m_i$ of the general curve $A$ at any of the points $p_i$ can be different from 0 (hence $m_i=3$ since $m_i\equiv 0 \mod 3$) only when $\alpha_j=0$ for all the points $q_j$.
\end{remark}

\begin{remark}\label{r.cover}
Assume 
$p_a(A')=g$. Then $A'K_Y=2g-2-{A'}^2$. On the other hand
\begin{align*}
3A'K_Y&=\st\pi{A'}\st\pi{K_Y}\overset{\eqref{e.Hur}}{=}\widetilde A(K_X-2\st\varepsilon{R_0}-2E-2F-2H)\\
&=(\st\varepsilon A-D)(\st\varepsilon{K_S-2R_0}+G-E)\\
&=AK_S-2AR_0-DG+DE
\end{align*}
Therefore
\begin{equation}\label{e.g}
AK_S-2AR_0-DG+DE=6g-6-3{A'}^2.
\end{equation}
\end{remark}

\begin{lemma}\label{l.DG}
In the above setting we have $DG=0$ unless $A^2=9$ and the general $A\in\mathcal A$ has an ordinary triple point at $q$. In the latter case $DG=-3$. In particular the general $A\in\mathcal A$ cannot have a cusp at $q$.
\end{lemma}

\begin{proof}
If ${\rm mult}_{q}A=0$ then obviously $DG=0$. Therefore we can assume $\alpha:={\rm mult}_{q}A\ge 1$. We notice that
\begin{equation}\label{e.AG}
3A'G'=\st\pi{A'}\st\pi{G'}=\widetilde AG=(\st\varepsilon A-D)G=-DG
\end{equation}
and then $DG\equiv 0 \mod 3$. Then simply compute $DG$ when $\alpha=1,2,3$ using lemmas \ref{l.drop1}, \ref{l.drop2} and \ref{l.drop3}. 
\end{proof}

As an immediate consequence of  lemma \ref{l.DG} and of equation \eqref{e.AG} we have 
\begin{cor}\label{c.AG}
In the above setting we have $A'G'=0$ unless $A^2=9$ and the general  $A$ has an ordinary triple point at $q$. In this latter case $A'G'=1$.
\end{cor}

We now concentrate our analysis on the case $\dim \Lambda=1$ and $h_2=1$, which is case (iii) of the list at page \pageref{list}.

\begin{prop}\label{p.list0} 
Assume $\dim \Lambda=1$ and $h_2=1$. Then when ${A'}^2=0$ one of the following possibilities holds:
\begin{enumerate}
\item[$(0a)$] $A^2=2$, $AR_0=0$, $g=1$, $A'K_Y=0$, $D=E_1+E_2$
\item[$(0b)$] $A^2=2$, $AR_0=1$, $g=1$, $A'K_Y=0$, $D=2F+G+H$
\item[$(0c)$] $A^2=3$, $AR_0=0$, $g=1$, $A'K_Y=0$, $D=E_1+E_2+E_3$
\item[$(0d)$] $A^2=3$, $AR_0=1$, $g=1$, $A'K_Y=0$, $D=E_1+2F+G+H$
\item[$(0e)$] $A^2=4$, $AR_0=0$, $g=1$, $A'K_Y=0$, $D=2E_1$
\item[$(0f)$] $A^2=5$, $AR_0=0$, $g=1$, $A'K_Y=0$, $D=2E_1+E_2$
\item[$(0g)$] $A^2=9$, $AR_0=0$, $g=2$, $A'K_Y=2$, $D=3F+3G+3H$
\item[$(0h)$] $A^2=9$, $AR_0=0$, $g=1$, $A'K_Y=0$, $D=3E_1$
\end{enumerate}

Moreover when cases $(0g)$ or $(0h)$ hold we have $\Phi=0$, i.e, the invariant pencil $\Lambda\le |3K_S|$ on $S$ has no fixed part.  
\end{prop}

\begin{proof}
Let us assume ${A'}^2=0$. We start by considering $A^2\le 7$. From lemma \ref{l.DG} we have $DG=0$.
Then, if $D=uF+vG+wH+\sum_{i=1}^{h_1}a_iE_i$, \eqref{e.g} becomes
\begin{equation}\label{e.g1}
AK_S-2AR_0-\sum_{i=1}^{h_1}a_i=6g-6
\end{equation}

Let us begin with $A^2=0$. Then $D=0$ and $\widetilde A=\st\varepsilon A$. Moreover from lemma \ref{le.sub} $AK_S=2$ and from remark \ref{r.R0} we have $0\le AR_0\le A\Phi=6$. Hence by \eqref{e.g1}
\begin{equation*}
2-2AR_0=6g-6
\end{equation*}
and then $AR_0=1,4$ since $AR_0\equiv 1 \mod 3$. The intersection cycle $A\cdot \Phi$ is composed of six  points with multiplicities. From remark \ref{r.R0} (we recall that we are assuming $R_0K_S=0$) these points are organized in orbits for the action of $\mathbb Z/3\mathbb Z$. Each orbit contains either three distinct points or only one fixed point, which can a priori be an isolated fixed point. 
The latter case cannot actually occur since $A$ and $\Phi$ have no isolated fixed point in common. Then we should have $AR_0\equiv 0 \mod 3$.  Contradiction.

When $A^2=1$ we have  $AK_S=3$, $AR_0\le A\Phi=8$ and from lemma \ref{l.drop1} $D=E_1$.
Then from remark \ref{r.L}  we have $D'\ge 2E_1$. Since $A\cdot \Phi$ is composed by 8 points with multiplicities and the only isolated fixed point in $A\cap \Phi$ is $p_1$, which is double for the 0-cycle $A\cdot \Phi$,  from remark \ref{r.R0} we have  $AR_0\equiv 0 \mod 3$. From \eqref{e.g1} we have
\begin{equation*}
2-2AR_0=3-2AR_0-1=6g-6
\end{equation*}
Then  $AR_0\equiv 1 \mod 3$ and this is impossible.
The rest of the proof when $A^2\le 7$ is similar.

Finally we consider $A^2=9$. We know from lemma \ref{le.sub} that $\Phi=0$. Using remark \ref{r.R0} we find $AR_0=0$. Moreover from remarks \ref{r.L} and \ref{r.A9}, $D'=0$ and the general $A$ has either multiplicity 0 or 3 at each of the isolated fixed points $p_j$, and it can be 3 only if the multiplicity at $q$ is 0.
 Then we have the following possibilities for $D$:\\
a) $D=3F+3G+3H$: in this case from lemma \ref{l.DG}  $DG=-3$  and \eqref{e.g} becomes
\begin{equation*}
6=3-0+3-0=AK_S-2AR_0-DG+DE=6g-6-3{A'}^2=6g-6
\end{equation*}
which has the only solution $g=2, A'K_Y=2$. \\
b) $D=3E_1$: equation \eqref{e.g} becomes
\begin{equation*}
0=3-0+0-3=AK_S-2AR_0-DG+DE=6g-6
\end{equation*}
which forces $g=1,A'K_Y=0$.
\end{proof}

With a similar argument one can show (see \cite[prop. 2.2.14, 2.2.15]{Pa}): 
\begin{prop}\label{p.list1} 
Assume $\dim \Lambda=1$ and $h_2=1$.  Then when ${A'}^2=1$ one of the following possibilities holds:
\begin{enumerate}
\item[$(1a)$] $A^2=3$, $AR_0=0$, $g=2$, $A'K_Y=1$, $D=0$
\item[$(1b)$] $A^2=3$, $AR_0=3$, $g=1$, $A'K_Y=-1$, $D=0$
\item[$(1c)$] $A^2=3$, $AR_0=6$, $g=0$, $A'K_Y=-3$, $D=0$
\item[$(1d)$] $A^2=5$, $AR_0=0$, $g=2$, $A'K_Y=1$, $D=2F+G+H$
\item[$(1e)$] $A^2=5$, $AR_0=3$, $g=1$, $A'K_Y=-1$, $D=2F+G+H$
\item[$(1f)$] $A^2=9$, $AR_0=0$, $g=2$, $A'K_Y=1$, $D=3F+2G+3H$
\end{enumerate}

Moreover when case $(1f)$ holds we have $\Phi=0$, i.e, the invariant pencil $\Lambda\le |3K_S|$ on $S$ has no fixed part.  
\end{prop}

\begin{prop}\label{p.list2} 
Assume $\dim \Lambda=1$ and $h_2=1$. The case  ${A'}^2=2$ cannot occur.
\end{prop}

\begin{remark}\label{r.N}
There is only one possibility left out by propositions \ref{p.list0}, \ref{p.list1} and \ref{p.list2}. This is the case $A^2=9, D=0$ or, equivalently, $A'=N$. Then from lemma \ref{le.N} we know $g=3, A'K_Y=NK_Y=1$.
\end{remark}

\begin{cor}\label{c.AH}
In the above setting, when $D=2F+G+H+\sum_i a_iE_i$ we find $A'H'=0$. 
\end{cor}

\section{Adjoint systems to the pencil $|N|$}\label{s.adj}

We also state here some properties of the adjoint system $|K_Y+N|$ which will be useful later. We know that $h^2(Y,\mathcal O_Y)=0$, so $Y$ is a regular surface, and that we have a linear system $|N|$ of nef and big curves on $Y$. Let us consider the short exact sequence
\begin{equation*}
0\to \mathcal O_Y(-N)\to \mathcal O_Y\to \mathcal O_N\to 0
\end{equation*}
Since $N$ is nef and big we have $h^0(Y, \mathcal O_Y(-N))=h^1(Y, \mathcal O_Y(-N))=0$ 
and
\begin{equation*}
h^0(Y, \mathcal O_Y(K_Y+N))=h^2(Y, \mathcal O_Y(-N))=h^1(N,\mathcal O_N)=p_a(N)=3-R_0K_S
\end{equation*}
Then $|N+K_Y|$ is a linear system of curves with arithmetic genus given by the formulas (see also lemma \ref{le.N}) 
\begin{eqnarray}\label{e.n1}
(N+K_Y)^2=N^2+K_Y^2+2NK_Y=5+K_Y^2-4R_0K_S\\
(N+K_Y)K_Y=K_Y^2+NK_Y=K_Y^2+1-2R_0K_S\notag\\
p_a(N+K_Y)=1+\frac{(N+K_Y)(N+2K_Y)}2=4+K_Y^2-3R_0K_S\notag
\end{eqnarray}

\begin{remark}\label{r.nef}
We observe that $N+K_Y$ is not nef. In fact $(N+K_Y)G'_i=K_YG'_i=-1$.
\end{remark}

From \cite[lemma 2.2]{CCM2} if $N+K_Y$ is not nef then every irreducible curve $Z$ such that $Z(N+K_Y)<0$ is a $(-1)$-curve with $ZN=0$. By contracting the curves and repeating the above argument we can see that after contracting each  $(-1)$-cycle on $Y$ such that $ZN=0$ we get a surface  on which $N$ and its adjoint are both nef divisors.

\begin{lemma}\label{le.Z}
The number $n$ of $(-1)$-cycles $Z$ on $Y$ different from the ones of $G'$ for which $ZN=0$  is greater or equal than
\begin{equation*}
\frac{35}6R_0K_S-\frac 3 2R_0^2-\frac{10+2h_2}3. 
\end{equation*}
\end{lemma}

\begin{proof}
Let $Z$ be such a cycle. Then for any other $(-1)$-cycle $Z'$ that does not intersect $N$  we have $ZZ'=0$ by the Index theorem. In particular $Z$ does not intersect any curve $G'_i$.
Then
\begin{equation}\label{e.zn}
0=ZN=Z(3K_Y+2B_0+E'-3G')=-3+2B_0Z+E'Z
\end{equation}
and there is a $(-3)$-curve $E'_i$ intersecting $Z$ positively.
Moreover since $E'_iN=0$ we have $(Z\pm E'_i)^2<0$
hence $-1\le ZE'_i\le 1$ for all $i=1,\dots,h_1$.

We know that $N+K_Y-\sum_{i=1}^n Z_i-G'$ is nef and, by lemma \ref{le.N} and equation \eqref{e.KY2},
\begin{equation*}
0\le (N+K_Y-\sum_{i=1}^n Z_i-G')^2=\frac{10+2h_2}3+\frac3 2 R_0^2-\frac{35}6R_0K_S+n.
\end{equation*}
\end{proof}

Let us set $N_1:=N+K_Y-G'-\sum_{i=1}^n Z_i$. 

\subsection{The $(-1)$-cycles $Z_i$}

We now analyse the irreducible components of the above $(-1)$-cycles $Z_i$. From the nefness of $N$ and $N_1$ we find
\begin{prop}\label{p.-1cycles}
In the above setting each irreducible component of the $(-1)$-cycles $Z$ is a curve $C$ such that $CN=CN_1=0$.
\end{prop}

\begin{cor}\label{c.F}
The curves $F'_j$ and $H'_j$ satisfy $F'_jN_1=H'_jN_1=0$ for any $j=1,\dots,h_2$. In particular $F'_j\sum_{i=1}^n Z_i=H'_j\sum_{i=1}^n Z_i=0$.
\end{cor}

\begin{proof}
Simply compute $F'_jN_1$ (resp. $H'_jN_1$) using the above proposition and  recalling that $F'_j$ and $H'_j$ are $(-3)$-curves such that $F'_jN=H'_jN=0$.
\end{proof}

\begin{cor}\label{c.E}
For any irreducible curve $E'_k$ or $B_{0k}$ we find
\begin{equation*}
E'_k\sum_{i=1}^nZ_i\ge 0, \quad B_{0k}\sum_{i=1}^nZ_i\ge 0.
\end{equation*}
\end{cor}

\begin{proof}
The statement is obvious if $E'_k$ or $B_{0k}$ are not contained in any of the $(-1)$-cycles $Z_i$. On the other hand if $E'_k$ is contained in some $(-1)$-cycles, from proposition \ref{p.-1cycles} we find $0=E'_kN_1=1-E'_k\sum_{i=1}^nZ_i$ hence $E'_k\sum_{i=1}^nZ_i=1$.
Analogously if $B_{0k}$ is contained is some cycle $Z_{i_0}$, then $B_{0k}N=0$ and  $B_{0k}$ is a $(-6)$-curve on $Y$. Hence we find $0=B_{0k}N_1=4-B_{0k}\sum_{i=1}^nZ_i$  
as wanted.
\end{proof}

Let us now consider an irreducible $(-1)$-curve $C$ in a cycle.
Recall, from the proof of lemma \ref{le.Z} and from \eqref{e.zn}, that there is a curve $E'_i$ such that $CE'_i=1$. 
On the other hand, $(\st\pi C)^2=3C^2=-3$ and, for each $E'_i$ with $CE'_i=1$ and $C$ as above, $3=\st\pi{CE'_i}=3\st\pi C E_i$ which implies $\st\pi CE_i=1$.
Moreover both $C$ and $E'_i$ are irreducible and $C^2=-1$ while ${E'_i}^2=-3$. Thus it cannot be $E_i\le \st\pi C$. 
In particular $\st\pi C$ cannot be singular at the point $\st\pi C\cap E_i$ (otherwise we should have $\st\pi CE_i\ge 2$).

\begin{lemma}\label{l.-3Z}
If $C$ is an irreducible $(-1)$-curve such that $CN=0$ and $CF'_j=CH'_j=0$  ($j=1,\dots,h_2$) then $\st\pi C$ is a rational curve and $\st\pi C^2=-3$.
\end{lemma}

\begin{proof}
Suppose that $\st\pi C=C_1+C_2+C_3$ is the union of three distinct curves and consider the curve $E_i$ above. Then $C_iC_j\ge 1$ for $i\ne j$ since the point $\st\pi Z\cap C_i$ is fixed for $\sigma$. 
Each component of $\st\pi C$ is a rational curve, so $-2=2p_a(C_i)-2=C_i(C_i+K_X)$.
Since the intersection of the components $C_i$ is fixed under the action of $\sigma$, we should have, for each $i$ such that $E'_i$ intersects $C$, $1=\st\pi CE_i=(C_1+C_2+C_3)E_i=3C_1E_i$.
Contradiction. It follows that $\st\pi C$ is an irreducible curve. We now want to show that $p_g(\st\pi C)=0$.
From Hurwitz formula we have
\begin{equation*}
2p_g(\st\pi C)-2=-2\cdot 3+2r
\end{equation*}
where $r$ is the number of ramification points of the triple cover $\st\pi C^\nu\longrightarrow C$. We have $2r=2p_g(\st\pi C)+4\ge 4$, so $r\ge 2$. On the other hand $r$ is not greater than the number of intersection points of $C$ with $B_0+E'+F'+H'$. We have $CF'=CH'=0$  and from \eqref{e.zn} either $CB_0=CE'=1$
or $CB_0=0,CE'=3$. Furthermore
\begin{equation*}
\st\pi CK_X=\st\pi C(\st\pi{K_Y}+2R)=-3+2CB_0+2CE'.
\end{equation*}
 In the former case, $r=2$ and $\st\pi C^\nu$ is a smooth rational curve. In the latter case $r=2,3$. If $r=2$ then $\st\pi C$ has geometric genus 0 and it has a singular point in $\st\pi C\cap E$. 
This is a contradiction since $\st\pi CE_i\le 1$, $i=1,\dots,h_1$. 

 When $r=3$, instead, since $p_a(\st\pi C)=p_g(\st\pi C)=1$, $\st\pi C$ should be a smooth elliptic curve.
When we look at the image $\varepsilon(\st\pi C)$ of this curve on $S$, since $CE'=3$ (recall from the proof of lemma \ref{le.Z} that $0\le CE'_i\le 1$ for any $i=1,\dots,h_1$) we would have 
$\varepsilon(\st\pi C)^2=\st\pi C^2-3=0$, and since it is an elliptic curve,  
$K_S\,\varepsilon(\st\pi C)=0$. 
This is impossible since $S$ is a minimal surface of general type.
 \end{proof}

\begin{cor}\label{c.Zint}
For any curve $C$ as above $CB_0=CE'=1$.
\end{cor}

We now want to determine the composition of the reducible $(-1)$-cycles. 

\begin{lemma}\label{l.noG} 
The curves $G'_j$ cannot be contained in one of the cycles $Z_i$, $i=1,\dots,n$. 
\end{lemma}

\begin{proof}
If one of the cycles $Z_i$, say $Z_{i_0}$  contains a curve $G'_j$ then from corollary \ref{c.F} we have, since $F'_jG'_j=1$,
\begin{equation*}
0=F'_j\sum_{i=1}^n Z_i=1+F'_j(Z_{i_0}-G'_j)+F'_j\sum_{i\ne i_0}Z_i
\end{equation*}
hence $F'_j$ is contained either in $Z_{i_0}$ or in another cycle $Z_i$ with $i\ne i_0$. In this latter case we have 
\begin{equation*}
0=G'_jZ_i=G'_j(Z_i-F'_j)+G'_jF'_j=G'_j(Z_i-F'_j)+1
\end{equation*}
hence $G'_j$ is also contained in $Z_i$.
Then there exists a cycle containing both $G'_j$ and $F'_j$. The same argument holds for $H'_j$. In particular  $F'_j$ and $H'_j$ are both contracted  to make the adjoint divisor to $N$ a nef divisor.

When we contract the curve $G'_j$ the images of $F'_j$ and $H'_j$ are two $(-2)$-curves meeting  at one point. Since they are both contracted there is a $(-1)$-cycle $C$ intersecting at least one of them at one point. If  $C$ passes through the intersection point of the $(-2)$-curves, then by contracting $C$ we obtain a cycle which is composed of two $(-1)$-curves meeting at one point. In particular this cycle is effective with self-intersection 0 and  it does  not intersect the image $\bar N$ of $N$ contradicting the Index theorem. 
This implies that $C$ is a $(-1)$-cycle  intersecting at one point only one of the curves $F'_j$ or $H'_j$. We will assume without loss of generality $CF'_j=1$.

We show the lemma by reducing ourselves to the case when $C$ is an irreducible $(-1)$-curve hence $C=Z_1$. This is always possible after the contraction of a suitable number of $(-1)$-curves. In this case  we have the configuration of figure \ref{f.1}
\begin{figure}
\begin{equation*}
\xymatrix{ 
\ar@{-}_{Z_1}[ddd] \\
&\\
&\\
&\\
}
\quad
\xymatrix{ 
&\ar@{-}_{Z_1}[ddd] & &\ar@{-}^{G'_j}[ddd] &\\
&  & & & &\\
\ar@{-}_{F'_j}[rrrr]& & & & \\
& & & & &\\
}
\end{equation*}
\begin{equation*}
\xymatrix{ 
&\ar@{-}_{Z_1}[ddd] & &\ar@{=}^{G'_j}[ddd] &\\
&  &\ar@{-}^{H'_j}[rrr] & & &\\
\ar@{-}_{F'_j}[rrrr]& & & & \\
& & & & &\\
}
\end{equation*}
\caption{}\label{f.1}
\end{figure}
hence $n\ge 3$. Moreover we have
\begin{equation*}
Z_1N=Z_1(3K_Y+2B_0+E'-3G')=-3+2B_0Z_1+E'Z_1.
\end{equation*}
Since $Z_1$ is irreducible then either $B_0Z_1=0, E'Z_1=3$ or $B_0Z_1=1, E'Z_1=1$. By the Index theorem, since $E'_kN=Z_iN=0$ for all $k=1,\dots,h_1, i=1,\dots,n$, we have $-1\le E'_kZ_i\le 1$. For any curve $E'_k$ such that $E'_kZ_1=1$ we find (see also corollary \ref{c.E})
\begin{equation*}
0\le E'_k\sum_{i=1}^n Z_i=1-E'_kN_1\le 1
\end{equation*}
and 
\begin{equation*}
E'_k\sum_{i=1}^n Z_i=3E'_kZ_1+E'_k(\sum_{i\ge 4}Z_i)=3+E'_k(\sum_{i\ge 4}Z_i)\le 1.
\end{equation*}

Thus $E'_k$ is contained in some $(-1)$-cycle $Z_i$ $i\ge 4$. Then $E'_k$ is contracted too and one of the  cycles has the configuration of figure \ref{f.2}.
\begin{figure}
\begin{equation*}
\xymatrix{ 
 & & & & & & &\\
 & &\ar@{=}_{Z_1}[ddd] & & &\ar@{=}^{G'_j}[ddd] &\\
 & &  &\ar@{-}_{E'_k}[lll] &\ar@{-}^{H'_j}[rrr] & & &\\
 &\ar@{-}_{F'_j}[rrrrr]& & & & & \\
 & & & & & & &\\
}
\end{equation*}
\caption{}\label{f.2}
\end{figure}
When we contract the curves $Z_1$ and $G'_j$  the images of $E'_k$ and $H'_j$ are $(-2)$-curves while the image of $F'_j$ is a $(-1)$-curve intersecting them at one point. Hence when we contract that $(-1)$-curve we obtain two $(-1)$-curve meeting at one point. This new configuration has self-intersection 0 and cannot be contracted to a point.
Thus we get a contradiction and the curve $G'_j$ cannot be contained in a cycle.\end{proof}

\begin{cor}\label{c.noG}
The curves $F'_j$ and $H'_j$ are not contained in any of the $(-1)$-cycles $Z_i$.
\end{cor}

\begin{proof}
If a curve $F'_j$ (or $H'_j$) is contained in a cycle $Z_i$ then, since $G'_jZ_i=0$, 
$G'_j$ is also contained in $Z_i$. This contradicts lemma \ref{l.noG}. 
\end{proof}

\begin{lemma}\label{l.no2E} 
There is no cycle $Z_i,1\le i\le n$ containing at least two curves $E'_k$.  
\end{lemma}

\begin{proof}
Let us assume that two of the curves $E'_k$, say $E'_1$ and $E'_2$, are contained in a reducible cycle $Z_{i_0}$. Then $E'_kN_1=0$ implies $E'_k\sum_{i}Z_i=1$ and there are two $(-1)$-cycles $Z_1$ and $Z_2$ such that $E'_1Z_1=1$, $E'_2Z_2=1$. 

Then we have the configuration of figure \ref{f.3}
\begin{figure}
\begin{equation*}
\xymatrix{ 
\ar@{-}_{Z_1}[ddd] \\
&\\
&\\
&\\
}
\xymatrix{ 
\ar@{-}_{Z_2}[ddd] \\
&\\
&\\
&\\
}
\xymatrix{ 
 & &\ar@{-}_{E'_1}[ddd] & & &\ar@{-}^{E'_2}[ddd] &\\
 & &  &\ar@{-}_{Z_1}[lll] &\ar@{-}^{Z_2}[rrr] & & &\\
 &\ar@{--}^{C}[rrrrr]& & & & & \\
 & & & & & & &\\
}
\end{equation*}
\caption{}\label{f.3}
\end{figure}
where $C$ is a suitable cycle. One can easily see that, in order to contract $E'_1$ (and analogously $E'_2$), the configurations of figure \ref{f.7}
\begin{figure}
\begin{equation*}
\xymatrix{ 
&\ar@{--}_{C}[ddd] & &\ar@{-}^{Z_2}[ddd] &\\
&  & & & &\\
\ar@{-}_{E'_2}[rrrr]& & & & \\
& & & & &\\
}
\xymatrix{ 
&\ar@{-}_{Z_1}[ddd] & &\ar@{--}^{C}[ddd] &\\
&  & & & &\\
\ar@{-}_{E'_1}[rrrr]& & & & \\
& & & & &\\
}
\end{equation*}
\caption{} \label{f.7}
\end{figure}
are $(-1)$-cycles. Then we have 
\begin{equation*}
-1=(C+E'_2+Z_2)^2=C^2-3-1+2CE'_2+2=C^2+2CE'_2-2
\end{equation*}
hence 
\begin{equation}\label{e.C}
C^2+2CE'_2=1
\end{equation}
and, analogously, $C^2+2CE'_1=1$.
Moreover from \eqref{e.C}
\begin{equation*}
-1=Z_3^2=(Z_1+E'_1+C+E'_2+Z_2)^2
=-3+2CE'_2.
\end{equation*}

Thus $CE'_2=1=CE'_1, C^2=-1$.
Then $C$ is a $(-1)$-cycle not intersecting $N$ hence by the Index theorem   we should have $CZ_3=0$. But $CZ_3=C(Z_1+E'_1+C+E'_2+Z_2)=1$ and we get a contradiction.
\end{proof}

\begin{cor}\label{c.red} 
If there is a reducible cycle $Z_{i_0}$, then $n\ge 3$ and for $n=3$ we have one of the following possibilities:
\begin{enumerate}
\item two irreducible $(-1)$-curves $Z_1$ and $Z_2$ and $Z_3=Z_1+Z_2+E'_k$ where $E'_kZ_1=E'_kZ_2=1$;
\item only one irreducible $(-1)$-curve $Z_1$, $Z_2=Z_1+C$, $Z_3=C+2Z_1+E'_k$ where $C$ is a $(-2)$-curve intersecting $Z_1$ at one point and $E'_k$ is such that $E'_kZ_1=1$.  
\end{enumerate}
\end{cor}

\begin{proof}
From lemmas \ref{l.noG} and \ref{l.no2E} if $n\le 2$ a reducible $(-1)$-cycle can contain at most one curve $E'_k$ and it does not contain any curve $G'_j$. Hence there is at least an irreducible curve $Z_1$. Then for $n=1$ the result is proved. For $n=2$ if $Z_2$ was reducible then $Z_2\ge Z_1$.
Then there exists a curve $E'_k$ intersecting $Z_1$ at one point and 
\begin{equation*}
E'_k(Z_1+Z_2)=E'_k(2Z_1+(Z_2-Z_1))=2+E'_k(Z_2-Z_1)\le 1.
\end{equation*}
This is only possible when $E'_k$ is contained in $Z_2$. But then there is at least another $(-1)$-cycle $Z_3$ intersecting $E'_k$ at one point and such that 
$Z_3N=0$ contradicting the assumption $n=2$.

When $n=3$ we can apply the above argument and we can see that if $Z_3$ is a reducible cycle then there is at least an irreducible  $(-1)$-cycle $Z_1$. Hence,  we have one of  the following configurations
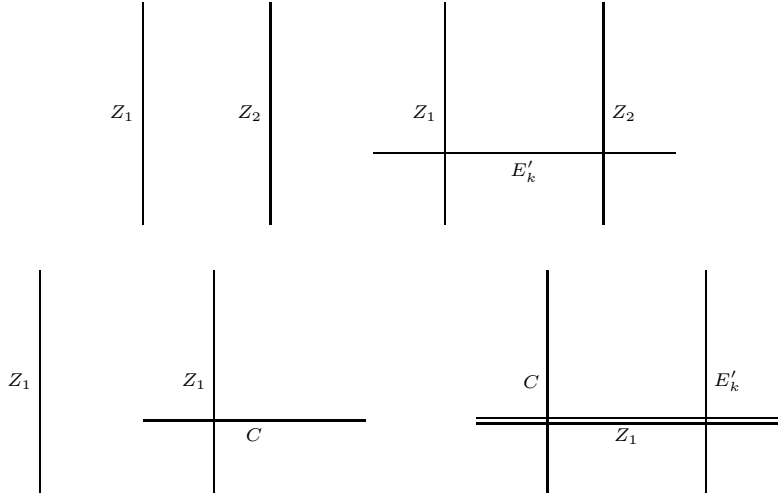
\begin{equation*}
\xymatrix{ 
\ar@{-}_{Z_1}[ddd] \\
&\\
&\\
&\\
}
\xymatrix{ 
\ar@{-}_{Z_2}[ddd] \\
&\\
&\\
&\\
}
\xymatrix{ 
&\ar@{-}_{Z_1}[ddd] & &\ar@{-}^{Z_2}[ddd] &\\
&  & & & &\\
\ar@{-}_{E'_k}[rrrr]& & & & \\
& & & & &\\
}
\end{equation*}
\begin{equation*}
\xymatrix{ 
\ar@{-}_{Z_1}[ddd] \\
&\\
&\\
&\\
}
\xymatrix{ 
&\ar@{-}_{Z_1}[ddd] & &\\
& & & &\\
\ar@{-}_C[rrr]& & &\\
& & & &\\
}
\xymatrix{ 
&\ar@{-}_{C}[ddd] & &\ar@{-}^{E'_k}[ddd] &\\
&  & & & &\\
\ar@{=}_{Z_1}[rrrr]& & & & \\
& & & & &\\
}
\end{equation*}
where $C$ is an irreducible $(-2)$-curve.
\end{proof}

\subsection{The linear systems $|N_i|$}

We now compute the arithmetic genus of $N_1$: from equation \eqref{e.n1} 
we know that $N_1^2=5-4R_0K_S+K_Y^2+n+h_2$ while from lemma \ref{le.N}
\begin{equation*}
N_1K_Y=(N+K_Y-\sum_{i=1}^n Z_i-G')K_Y=1-2R_0K_S+K_Y^2+n+h_2
\end{equation*}
so
\begin{equation}\label{e.M}
p_a(N_1)=1+\frac{N_1^2+N_1K_Y}2=4-3R_0K_S+K_Y^2+n+h_2 \le N_1^2
\end{equation}
since $0\le R_0K_S \le 1$. $|N_1|$ is again a linear system of nef curves, so when $p_a(N_1)\ge 1$ we can apply the same argument as in page \pageref{r.nef} to study the adjoint system $|N_1+K_Y|$. Under this hypotheses $N_1$ is nef and big and we find 
$h^0(Y,\mathcal O_Y(N_1+K_Y))=p_a(N_1)\ge 1$.
$N_1+K_Y$ is not nef since the curves $Z_i$ and $G'_i$ do not intersect $N_1$, but there could be some other $(-1)$-cycles $Z'_i$ such that  $Z'_iN_1=0$ (see \cite[lemma 2.2]{CCM2}). Then $N_2:=N_1+K_Y-\sum_{i=1}^n Z_i-G'-\sum_{j=1}^{n'} Z'_j$ is nef and we can compute $N_2^2, N_2K_Y, p_a(N_2)$. 

By repeating the same argument again (if $p_a(N_2)\ge 1$) we obtain the following proposition (see \cite[section 2.3.2]{Pa}):
\begin{prop}\label{p.comp}
In the above setting let us set $N_0:=N$. Then the numerical data of the curves $N_1,N_2,N_3$ are:
\begin{equation*}
\begin{tabular}{|c|c|c|}
\hline
  &$i=1$  &$i=2$\\
\hline
$N_i^2$ &$5-4R_0K_S+K_Y^2+n+h_2$ &$7-8R_0K_S+4K_Y^2+4n+4h_2+n'$\\
$N_iK_Y$ &$1-2R_0K_S+K_Y^2+n+h_2$ &$1-2R_0K_S+2K_Y^2+2n+2h_2+n'$\\
$p_a(N_i)$ &$4-3R_0K_S+K_Y^2+n+h_2$ &$5-5R_0K_S+3K_Y^2+3n+3h_2+n'$\\
$N_{i-1}N_i$ &$4-2R_0K_S$ &$6-6R_0K_S+2K_Y^2+2n+2h_2$\\
\hline
\end{tabular}
\end{equation*}
\begin{equation*}
\begin{tabular}{|c|c|}
\hline
   &$i=3$\\
\hline
$N_i^2$ &$9-12R_0K_S +9K_Y^2+9h_2+9n+4n'+n''$ \\
$N_iK_Y$ &$1-2R_0K_S +3K_Y^2+3h_2+3n+2n'+n''$ \\
$p_a(N_i)$ &$6-7R_0K_S+6K_Y^2+6h_2+6n+3n'+n''$ \\
$N_{i-1}N_i$ &$8-10R_0K_S+6K_Y^2+6h_2+6n+2n'$\\
\hline
\end{tabular}
\end{equation*}
\end{prop}

\section{Case (i): $R_0K_S=1$, $h_2=3$}\label{s.i}

We now prove the non-existence of case (i) (cf. the list of page \pageref{list}) by studying the pencil $|N_1|$ and by showing that the induced map  $\phi_{|N_1|}:Y\dasharrow \mathbb P_1$ has too many singular fibres. 
\begin{lemma}\label{l.h0}
In case (i) we find 
\begin{enumerate}
\item[(a)] $h_1\ge 1$ 
\item[(b)] no positive multiple of $K_S-2R_0$ is an effective divisor.
\end{enumerate}
\end{lemma}

\begin{proof}
(a) Since $R_0K_S=1$, there exists a unique irreducible component $\Gamma$ of $R_0$ for which $\Gamma K_S=1$. Using the Index theorem we find $\Gamma^2\le 1$. The other irreducible components of $R_0$ are $(-2)$-curves. Then 
$R_0^2\le \Gamma^2\le 1$.
Since $h_2=3$  from \eqref{e.fixed} we find
$h_1=\frac{3R_0K_S-R_0^2}2\ge 1$. 

To prove (b) note that $K_S(K_S-2R_0)=-1$ while $K_S$  is nef.
\end{proof} 

\begin{lemma}\label{l.k2}
Suppose  case (i) holds and  $R_0$ is the disjoint union of an irreducible component $\Gamma$ with $\Gamma K_S=1$ and of $\ell$  $(-2)$-curves. Then
\begin{enumerate}
\item[(a)] $K_Y^2=-4-3\ell+\frac{3\Gamma^2-1}2\le -3$.
\item[(b)] $0\le\ell\le (5+\Gamma^2)/2$
\item[(c)] $K_Y^2\ge -12$
\item[(d)] $h_1\le 4$.
\end{enumerate}
 \end{lemma}

\begin{proof}
(a) It is an easy computation which uses formula \eqref{e.KY2} of page \pageref{e.KY2}:
\begin{equation*}
K_Y^2=\frac 1 3(K_S^2-6-h_2-\frac{11}2R_0K_S+\frac 9 2R_0^2)=-4-3\ell+\frac{3\Gamma^2-1}2.
\end{equation*}

(b) Since $\pi:X\longrightarrow Y$ is a surjective map we have an injection\\ $H^2(Y,\mathbb C)\longrightarrow H^2(X,\mathbb C)$. In particular we find $e(X)=12-K_X^2\ge e(Y)=12-K_Y^2$ hence $K_Y^2\ge K_X^2$. Thus from lemma \ref{l.k2} and lemma \ref{l.h0}
\begin{equation*}
-4-3\ell+\frac{3\Gamma^2-1}2=K_Y^2\ge K_X^2=K_S^2-(h_1+3h_2)=1-\frac{3-\Gamma^2}2-\ell-9
\end{equation*}
hence
\begin{equation*}
2\ell\le -4+9-1+\frac{3\Gamma^2-1}2+\frac{3-\Gamma^2}2=5+\Gamma^2
\end{equation*}
as wanted.

(c) and (d) follow from (a), (b) and lemma \ref{l.h0}.
\end{proof}

\begin{prop}
Assume case (i) holds. Then $Y$ is a rational surface.
\end{prop}

\begin{proof}
By Castelnuovo's criterion, since $q(Y)\le q(X)=q(S)=0$, we need to show that $P_2(Y)=h^0(Y,\mathcal O_{Y}(2K_Y))=0$.
Since $R=\st\varepsilon{R_0}+E+F+H$, from \eqref{e.Bl} we have $\st\pi{2K_Y-2G'}=2K_X-4R-2G=\st\varepsilon{2K_S-4R_0}-2E$.

Moreover since $K_YG'_i=-1<0$, we find 
\begin{equation*}
0 \le P_2(Y)=h^0(Y,\mathcal O_Y(2K_Y-2G'))\le h^0(X,\mathcal O_X(\st\varepsilon{2K_S-4R_0}-2E)) 
\end{equation*}

But using lemma \ref{l.h0} we find 
\begin{equation*}
h^0(X,\mathcal O_X(\st\varepsilon{2K_S-4R_0}-2E))\le h^0(S,\mathcal O_S(2K_S-4R_0))=0
\end{equation*}
and then
$h^0(Y,\mathcal O_Y(2K_Y))=h^0(Y,\mathcal O_Y(2K_Y-2G'))=0$.
\end{proof}

From now on we use the same notation as in lemma \ref{le.Z}. From proposition \ref{p.comp} we have 
\begin{equation}\label{e.pan1}
p_a(N_1)=4-3R_0K_S+K_Y^2+n+h_2=4+K_Y^2+n=N_1^2
\end{equation}
\begin{equation}\label{e.nn1}
NN_1=4-2R_0K_S=2
\end{equation}

\begin{lemma}\label{l.n1}
In the above setting we have $N_1^2=0,1$.
\end{lemma}

\begin{proof}
Since  $N^2=3$ we have by the Index theorem and \eqref{e.nn1} 
$0\ge (3N_1-2N)^2=9N_1^2-12$ which implies $N_1^2=0,1$.
\end{proof}

Let us write $|N_1|=|\Delta|+T$, where $|\Delta|$ is the movable part and $T$ is the fixed part of $|N_1|$. Since $N_1$, $\Delta$, $N$ are nef divisors, we have 
$0\le \Delta N\le N_1N=2$.  
In particular it cannot be  $\Delta N=0$ otherwise, by the Index theorem 
and the rationality of $Y$, $\Delta=0$ whereas $h^0(Y,\mathcal O_Y(\Delta))=h^0(Y,\mathcal O_Y(N_1))=2$. Thus actually $1\le \Delta N\le N_1N=2$.

\begin{lemma}\label{l.N10}
Suppose $N_1^2=0$. Then $|N_1|$ has no fixed part.
\end{lemma} 

\begin{proof}
We have $0=N_1^2=N_1\Delta+N_1T$ or, equivalently,
\begin{eqnarray*}
0=N_1\Delta={\Delta}^2+\Delta T\\
0=N_1T=\Delta T+T^2
\end{eqnarray*}
which implies ${\Delta}^2=\Delta T=T^2=0$.

It cannot  be $N\Delta=1=NT$: we obtain by the Index theorem
$0\ge(\Delta-T)^2={\Delta}^2+T^2-2\Delta T=0$ which implies 
$\Delta\sim T$. Hence, since $Y$ is a rational surface, 
$\Delta\equiv T$ which is impossible.

So $N\Delta=NN_1=2$ and then $0\ge (N_1-\Delta)^2=T^2=0$.
Again, by the rationality of $Y$ we have $T\equiv 0$ and $|N_1|$ has no fixed part.
\end{proof}

\begin{lemma}\label{l.N1}
Suppose $N_1^2=1$. Then $|N_1|$ has no fixed part unless ${\Delta}^2=0, p_a(\Delta)=p_a(N_1)=1$ and either $\Delta N=1, N=N_1+\Delta$ or $\Delta N=2, N_1=\Delta+Z_i$ for some reducible $(-1)$-cycle $Z_i$.
\end{lemma}

\begin{proof}
We know that 
$1=N_1^2=N_1\Delta+N_1T$. It cannot be $N_1\Delta=0$, otherwise by the Index theorem   and the rationality of $Y$ it should be $\Delta=0$, which is impossible.

Then we have $N_1\Delta=1$, $N_1T=0$, and this implies $T^2\le 0$. When $T^2=0$ we see that $|N_1|$ has no fixed part, as wanted, whereas when $T^2$ is strictly negative, by
\begin{eqnarray}
1=N_1\Delta={\Delta}^2+\Delta T\label{e.t1}\\
0=N_1T=\Delta T+T^2\label{e.t2}
\end{eqnarray}
we find $\Delta T=1$, $T^2=-1$, ${\Delta}^2=0$. Then by \eqref{e.pan1} 
$N_1^2=\Delta^2+T^2+2\Delta T=1=p_a(N_1)$ and $N_1K_Y=-N_1^2=-1$. We now look at $N\Delta$.

If $N\Delta=2$ we have
\begin{equation*}
1=N_1\Delta=(K_Y+N-\sum_{i=1}^n Z_i-G')\Delta=\Delta K_Y+2-(\sum_{i=1}^n Z_i+G')\Delta
\end{equation*}
which amounts to say $0\le (\sum_{i=1}^n Z_i+G')\Delta=1+\Delta K_Y$.
So the followings are true: \\
a) $\Delta K_Y\ge 0$ \\
b) there exists a $(-1)$-curve $C$ (which can be either one of the $Z_i$'s or one of the $G'_i$'s)  intersecting $\Delta$ positively.

Using b) and the fact that $CN=TN=0$, the Index theorem   implies
$0\ge (T-C)^2=-2-2TC$.
Then $TC\ge -1$, whereas $CN_1=0$ implies $CT\le -1$. This forces $TC=-1$, $T\equiv C$. 
But for any irreducible such curve $C$ there is a curve $D$ such that $DN_1=0$ intersecting $C$ at one point, i.e. $D=E'_j$ for the $Z_i$ and $D=F'_i$ for $G'_i$ (see corollary \ref{c.Zint}), so that
$0=DN_1=D\Delta+DC=D\Delta+1$
which contradicts the nefness of $\Delta$.
Then $N_1\equiv\Delta+Z_i$ with $Z_i$ a reducible $(-1)$-cycle.

We are now left with the case $N\Delta=1$. Since $NN_1=2$ we have $NT=1$. Moreover $N_1\Delta=1$ and  ${\Delta}^2=0$ by \eqref{e.t1} and \eqref{e.t2}. From $(N_1+\Delta)N=3=N^2$ we have
\begin{equation*}
0\ge (N_1+\Delta-N)^2=N_1^2+{\Delta}^2+N^2+2N_1\Delta-2N_1N-2N\Delta=0
\end{equation*}
hence $N\equiv N_1+\Delta\equiv 2\Delta+T$.
Then $\Delta K_Y=(N-N_1)K_Y=0$ implies $p_a(\Delta)=1$ as wanted.
\end{proof}

\begin{cor}\label{c.n}
We have $n=N_1^2+3\ell+(1-3\Gamma^2)/2\le N_1^2+8\le 9$.
\end{cor}

\begin{proof}
From lemma \ref{l.k2} we know that $K_Y^2\ge -12$. From \eqref{e.pan1} and lemma \ref{l.n1} we find 
$n=N_1^2-4-K_Y^2\le N_1^2-4+12=N_1^2+8\le 9$.
Moreover, again from lemma \ref{l.k2}, $n=N_1^2-4-K_Y^2=N_1^2+3\ell+(1-3\Gamma^2)/2$.
\end{proof}

\begin{remark}\label{r.n}
From the above corollary 
$n=N_1^2+3\ell+(1-3\Gamma^2)/2\equiv N_1^2-1 \mod 3$.
In particular when $N_1^2=0$ we find $n\equiv 2\mod 3$ hence $n=2,5$ or $8$, while when $N_1^2=1$ we have $n\equiv 0 \mod 3$ hence $n=3,6$ or $9$.
\end{remark}

\begin{lemma}\label{l.fib}
In the above setting the pencil $|\Delta|$ determines a fibration $\phi_{|\Delta|}:Y\dasharrow\mathbb P^1$. Let us set 
$\delta:=\sum_s(e(\Delta_s)-e(\Delta))$
where the sum is taken over all the singular curves $\Delta_s\in|\Delta|$.
Then $\delta$ satisfies 
\begin{equation*}
18\le\delta =12+3N_1^2+n+{\Delta}^2\le 16+n.
\end{equation*}
In particular if $N_1^2=0$ then $n=8$.
\end{lemma}

Before proving the above lemma we need the following 
\begin{lemma}\label{l.-n}
Let $|C|$ be a pencil on a complex surface with $C^2=0$ and let $\phi_{|C|}$ be the induced fibration. An irreducible curve $C_1$ with $C_1^2=-n$ in a singular fibre contributes at least $n$ to the Euler number $\delta:=\sum_s(e(C_s)-e(C))$ of the fibration (see \cite[Proposition III.1.4]{BPV}).
\end{lemma}

\begin{proof}
Let us consider a reducible curve of the fibration $C=\sum_{i=1}^l h_iC_i$.
As shown in \cite{F} (see also \cite[section V.1]{E}) $C$ is equivalent to $\delta_0$ curves with a node where
\begin{equation}\label{e.fibre}
\delta_0\ge \sum_{i=1}^l (h_i-1)(2p_a(C_i)-2)+\sum_{i\ne j} (h_i+h_j-1)C_iC_j
\end{equation}

Let us consider one of the curves $C_j$, say $C_1$ with $C_1^2=-n$. Then
$0=C_1C=-nh_1+C_1\sum_{i=2}^l h_iC_i$ hence
$C_1\sum_{i=2}^l h_iC_i=nh_1$.
Since $C$ is connected we also have
$C_1\sum_{i=2}^l C_i\ge 1$.
 Then
\begin{align*}
\delta_0&\ge (h_1-1)(2p_a(C_1)-2)+\sum_{j\ge 2} (h_1+h_j-1)C_1C_j\\
&\ge (h_1-1)(-2+\sum_{j\ge 2}C_1C_j)+C_1\sum_{j\ge 2}h_jC_j\\
&\ge (h_1-1)(-1)+nh_1=(n-1)h_1+1\ge n
\end{align*}
as wanted. Since each node of a curve increases the Euler number by 1 the result is proved.
\end{proof}

\begin{proof}[Proof of lemma \ref{l.fib}]
Off its ${\Delta}^2$ base points, the pencil $|\Delta|$ determines on $Y$ a fibration over $\mathbb P^1$ of curves of genus $0\le p_a(\Delta)=N_1^2\le 1$. Computing Euler numbers from \cite[Proposition III.11.4]{BPV} we find
\begin{equation*}
12-K_Y^2+{\Delta}^2=e(Y)+{\Delta}^2=2(2-2N_1^2)+\sum_{s}(e(\Delta_s)-e(\Delta))
\end{equation*}
where $\Delta_s$ are the singular curves of $|\Delta|$. 
Let us set $\delta:=\sum_{s}(e(\Delta_s)-e(\Delta))$. 
Then since from \eqref{e.pan1} $K_Y^2=N_1^2-4-n$
\begin{equation*}
\delta=12-K_Y^2+{\Delta}^2-4+4N_1^2=12+3N_1^2+n+{\Delta}^2\le 16+n.
\end{equation*}

Let us first consider the curves $F'_i$ and $H'_i$. From lemma \ref{le.Z} we find
$F'_iN_1=0$ and the same holds for $H'_i$. 

Moreover if $N\equiv N_1+\Delta$ (see lemma \ref{l.N1}) we also find 
$0=F'_iN=F'_i(N_1+\Delta)=F'_i\Delta$ and, again, the same holds for $H'_i$. 

If $N_1\equiv \Delta+Z_i$ instead we find from corollary \ref{c.noG}
$0=F'_iN_1=F'_i(\Delta+Z_i)\ge F'_i\Delta$
hence $F'_i\Delta=H'_i\Delta=0$ in any case. 
Thus from lemma \ref{l.-n} each of the curves $F'_i$ and $H'_i$ contributes  3 to $\delta$, hence $\delta\ge 6h_2=18$.

If $N_1^2=0$ then $\Delta\equiv N_1$ and 
\begin{equation*}
18\le\delta=12+3N_1^2+n+{\Delta}^2=12+n
\end{equation*}
hence $n\ge 6$. In particular from remark \ref{r.n} we can deduce $n=8$.
\end{proof}

\begin{prop}\label{p.no0}
Case (i) cannot occur with $N_1^2=0$.
\end{prop}

\begin{proof}
Let us assume $N_1^2=0$. Then from lemma \ref{l.fib}
$18\le \delta=12+n=12+8=20$.

If all the eight cycles $Z_i$ were irreducible then we should have
$Z_iN_1=0$ hence from lemma \ref{l.-n} each of them would contribute 1  to $\delta$. Thus $18+n=26\le\delta=20$. Contradiction.

If one of the cycles is reducible, then from lemma \ref{l.noG}, corollary \ref{c.noG} and lemma \ref{l.no2E} there is a curve $E'_k$ contained in that cycle and from lemma \ref{l.-n} $E'_k$ increases $\delta$ by 3. Then
$18+3=21\le\delta=20$ 
and, again, we get a contradiction.
\end{proof}

\begin{prop}\label{p.noZ}
The case $|N_1|=|\Delta|+Z_i$ with $Z_i$ reducible $(-1)$-cycle cannot occur.
\end{prop}

\begin{proof}
From the proof of lemma \ref{l.N1} we know that $\Delta N=NN_1=2, \Delta G'=0$ and $\Delta^2=\Delta K_Y=0$. Hence, from the definition of $N$,
\begin{equation*}
2=N\Delta=\Delta(3K_Y+2B_0+E'-3G')=2B_0\Delta+E'\Delta.
\end{equation*}
Since $\Delta$ is nef we find either $B_0\Delta=1,E'\Delta=0$ or $B_0\Delta=0,E'\Delta=2$.  We recall that $B_0=\Gamma'+\sum_{i=1}^\ell B_{0i}$ with $B_{0i}^2=-6$, $B_{0i}\cong \mathbb P^1$.

We have to analyse separately  the two cases 
\begin{enumerate}
\item[\underline{Case I}] $B_0\Delta=1, E'\Delta=0$
\item[\underline{Case II}] $B_0\Delta=0,E'\Delta=2$.
\end{enumerate}

We only write the proof for case I. The other one is similar (see also \cite[prop. 3.1.16]{Pa}). 

\underline{Case I}: 
Since none of the $(-3)$-curves $E'_k$ intersects $\Delta$,  each of them contributes 3 to $\delta$ (see  lemma \ref{l.-n}). Moreover there is only one irreducible component of $B_0$ intersecting $\Delta$.

If $\Gamma'\Delta=1$ then we also have the contribution of $\ell$ irreducible $(-6)$-curves of $B_0$. Then, using lemma \ref{l.-n} and corollary \ref{c.n},
\begin{align*}
18+3h_1+6\ell&=18+3(\frac{3-\Gamma^2}2+\ell)+6\ell\le\delta=12+3N_1^2+n+{\Delta}^2\\
&=15+n=15+N_1^2+3\ell+\frac{1-3\Gamma^2}2=16+3\ell+\frac{1-3\Gamma^2}2
\end{align*}
hence
\begin{equation*}
2+6\ell\le \frac{1-3\Gamma^2}2+\frac{3\Gamma^2-9}2=-4
\end{equation*}
and we get a contradiction.

If $\Gamma'\Delta=0$ we have necessarily ${\Gamma'}^2\le 0$ (hence $\Gamma^2\le -1$ on the numerical Godeaux surface $S$) and there is a $(-6)$-curve $B_{0k}$ in $B_0$ such that $B_{0k}\Delta=1$. In particular $\ell\ge 1$. Then $\Gamma'$ contributes  $-{\Gamma'}^2=-3\Gamma^2$ to $\delta$ and we can also consider $\ell-1$ $(-6)$-curves $B_{0i}$, $i\ne k$, plus the $h_1$ curves $E'_k$. Thus
\begin{align*}
18-3\Gamma^2+6(\ell-1)+3h_1&=12+6\ell-3\Gamma^2+\frac{9-3\Gamma^2}2+3\ell\le\delta\\
&=15+n=16+3\ell+\frac{1-3\Gamma^2}2
\end{align*}
hence
\begin{equation*}
6\ell\le 4+\frac{1-3\Gamma^2}2+\frac{3\Gamma^2-9}2+3\Gamma^2=3\Gamma^2<0.
\end{equation*}
Contradiction.
\end{proof}

\begin{prop}\label{p.noN}
The case $N\equiv N_1+\Delta$  cannot occur.
\end{prop}

\begin{proof}
From the proof of lemma \ref{l.N1} we know that $\Delta N=1, NN_1=2, \Delta G'=0$ and $\Delta^2=\Delta K_Y=0$. Hence
\begin{equation*}
1=N\Delta=\Delta(3K_Y+2B_0+E'-3G')=2B_0\Delta+E'\Delta.
\end{equation*}
From the nefness of  $\Delta$ we find $B_0\Delta=0,E'\Delta=1$.
Since $\Gamma'\Delta=0$ we find ${\Gamma'}^2\le 0$ hence $\Gamma^2\le -1$ on $S$. In particular $h_1=\frac{3-\Gamma^2}2+\ell\ge 2$. All the irreducible components of $B_0$ and $h_1-1$ curves $E'_k$ contribute to $\delta$. Thus from lemma \ref{l.-n} and corollary \ref{c.n}
\begin{align*}
18-3\Gamma^2+6\ell+3(h_1-1)&=15+6\ell-3\Gamma^2+\frac{9-3\Gamma^2}2+3\ell\le\delta\\
&=12+3N_1^2+n+{\Delta}^2=15+n\\
&=15+N_1^2+3\ell+\frac{1-3\Gamma^2}2=16+3\ell+\frac{1-3\Gamma^2}2
\end{align*}
hence 
\begin{equation*}
6\ell\le 1+\frac{1-3\Gamma^2}2+\frac{3\Gamma^2-9}2-3\Gamma^2=-3+3\Gamma^2<0
\end{equation*}
and  we get a contradiction.
\end{proof}

Thus from lemma \ref{l.N1}, propositions \ref{p.noZ} and \ref{p.noN} we immediately find
\begin{cor}
In the above setting the pencil $|N_1|$ has no fixed part.
\end{cor}

We recall that from remark \ref{r.n}  when $N_1^2=1$ we have $n=3,6,9$. 

\begin{prop}\label{p.noN1}
Case (i) with $N_1^2=1$  can only  occur when $n=6$ and either $\Gamma^2=-3$ and $\ell=0$ or $\Gamma^2=-1$ and $\ell=1$.
In particular $\Gamma^2=1$ cannot occur.
 Moreover all the curves $E'_k$ intersect $N_1$ at one point.
\end{prop}

\begin{proof}
From the above corollary we have $\Delta\equiv N_1$.
Since from \eqref{e.pan1} $p_a(N_1)=N_1^2=1$  we know that $NN_1=2, N_1G'=0$ and $N_1^2=1, N_1K_Y=0$. Hence
\begin{equation*}
2=NN_1=N_1(3K_Y+2B_0+E'-3G')=-3+2B_0N_1+E'N_1.
\end{equation*}

Moreover, from proposition \ref{p.-1cycles}, 
\begin{equation*}
0\le E'_kN_1=E'_k(N+K_Y+G'-\sum_iZ_i)=1-E'_k\sum_iZ_i\le 1
\end{equation*}
hence from lemma \ref{l.k2} $0\le E'N_1\le h_1\le 4$. Then
either $B_0N_1=2,E'N_1=1$ or $B_0N_1=1,E'N_1=3$. 
As in the proof of proposition \ref{p.noZ} we only prove the first case (see also \cite[prop. 3.1.19]{Pa}).
 
\underline{Case I}: 
All the $(-3)$-curves $E'_k$, except for one, have  no intersection with  $N_1$ hence from lemma \ref{l.-n} each of them contributes  3 to $\delta$. Moreover there are at most two  irreducible components of $B_0$ intersecting $N_1$. 

If $\Gamma'N_1=2$ then we have the contribution of $\ell$ irreducible $(-6)$-curves of $B_0$. Then, using lemmas \ref{l.fib}, \ref{l.-n} and corollary \ref{c.n},
\begin{align*}
18+3(h_1-1)+6\ell&=18+3(\frac{3-\Gamma^2}2+\ell-1)+6\ell\le\delta\\
&=12+3N_1^2+n+{\Delta}^2=16+n\\
&=16+N_1^2+3\ell+\frac{1-3\Gamma^2}2=17+3\ell+\frac{1-3\Gamma^2}2
\end{align*}
hence
\begin{equation*}
6\ell\le 2+\frac{1-3\Gamma^2}2+\frac{3\Gamma^2-9}2=-2
\end{equation*}
and we get a contradiction.

If $\Gamma'N_1=1$ then there is a $(-6)$-curve $B_{0k}$ intersecting $N_1$ at one point. Hence we have the contribution of $\ell-1$ irreducible $(-6)$-curves of $B_0$. Thus, using lemmas \ref{l.fib}, \ref{l.-n} and corollary \ref{c.n},
\begin{align*}
18+3(h_1-1)+6(\ell-1)&=18+3(\frac{3-\Gamma^2}2+\ell-1)+6\ell-6\le\delta\\
&=16+n=17+3\ell+\frac{1-3\Gamma^2}2
\end{align*}
hence
\begin{equation*}
6\ell\le 8+\frac{1-3\Gamma^2}2+\frac{3\Gamma^2-9}2=4
\end{equation*}
which forces $\ell=0$ while we know $\ell\ge 1$. Contradiction.

If $\Gamma'N_1=0$ we have necessarily ${\Gamma'}^2\le 0$ (hence $\Gamma^2\le -1$ on the numerical Godeaux surface $S$) and there is at least one  $(-6)$-curve $B_{0k}$ in $B_0$ such that $B_{0k}N_1\ge 1$. In particular $\ell\ge 1$. Then $\Gamma'$ contributes $-{\Gamma'}^2=-3\Gamma^2$ to $\delta$. If $\ell\ge 2$ we can also consider $\ell-2$ $(-6)$-curves $B_{0i}$, $i\ne k$, plus the $h_1-1$ curves $E'_k$. Thus
\begin{align*}
18-3\Gamma^2+6(\ell-2)+3(h_1-1)&=3+6\ell-3\Gamma^2+\frac{9-3\Gamma^2}2+3\ell\le\delta\\
&=16+n=17+3\ell+\frac{1-3\Gamma^2}2
\end{align*}
hence
\begin{equation*}
6\ell\le 14+\frac{1-3\Gamma^2}2+\frac{3\Gamma^2-9}2+3\Gamma^2=10+3\Gamma^2\le 7
\end{equation*}
which forces $\ell\le 1$ contradicting the assumption $\ell\ge 2$.

If $\ell=1$ then we only have the contribution of $\Gamma'$ and of $h_1-1$ curves $E'_k$. Then
\begin{align*}
&18-3\Gamma^2+3h_1-3=18-3\Gamma^2+\frac{9-3\Gamma^2}2+3\ell-3\\
&=18-3\Gamma^2+\frac{9-3\Gamma^2}2\le\delta=17+3\ell+\frac{1-3\Gamma^2}2=20+\frac{1-3\Gamma^2}2
\end{align*}
which forces
\begin{equation*}
0\le 2+\frac{1-3\Gamma^2}2+\frac{3\Gamma^2-9}2+3\Gamma^2=-2+3\Gamma^2<-2.
\end{equation*}
Contradiction.
\end{proof}

We now show the following
\begin{prop}\label{p.no16}
Case (i) with $N_1^2=1$ and $n=6$ cannot occur.
\end{prop}

\begin{proof}
Let us assume $N_1^2=1$ and $n=6$. Then from lemma \ref{l.fib}
\begin{equation*}
18\le \delta=12+3N_1^2+n+{\Delta}^2\le 16+n=16+6=22.
\end{equation*}

If all the six cycles $Z_i$ were irreducible then  each of them would not intersect $N_1$ and $\Delta$. Then  from lemma \ref{l.-n} they would contribute  $1\cdot 6=6$ to $\delta$ hence $18+6=24\le\delta\le 22$ and we would  get a contradiction.

Let us assume there is at least one reducible cycle. Then  one of the irreducible $(-1)$-curves, say $Z_1$, appears with multiplicity $m_1\ge 2$ in $\sum_i Z_i$. For any curve $E'_k$ such that $E'_kZ_1=1$ we find
\begin{equation*}
E'_k\sum_i Z_i=E'_k(m_iZ_1+(\sum_i Z_i-m_iZ_1))=m_i+E'_k(\sum_i Z_i-m_iZ_1)
\end{equation*}
It follows that $E'_k$ is contained in some cycle $Z_i$, $i\ge 2$ and then from proposition \ref{p.-1cycles} $E'_kN_1=0$ contradicting proposition \ref{p.noN1}.
\end{proof}

Propositions \ref{p.no0}, \ref{p.noZ}, \ref{p.noN}, \ref{p.noN1} and \ref{p.no16} can be summarized in the following theorem.
\begin{theorem}\label{t.i}
Case (i) cannot occur.
\end{theorem}

\section{Case (ii): $R_0K_S=0$, $h_2=4$}\label{s.ii}

In this section we show that also case (ii) cannot occur by studying the map  $\phi_{|M'|}:Y\longrightarrow \mathbb P^1$ where $|M'|$ is the image on $Y$ of the movable part of $|2K_S|$.

Assume case (ii) holds. From proposition \ref{p.rk} and formula \eqref{eq.split1} we have 
\begin{equation*}
2=h^0(Y,\mathcal O_Y(2K_Y+B))\le h^0(X,\mathcal O_Y(2K_X-R))\le h^0(X,\mathcal O_X(2K_X))=2
\end{equation*}
 which implies that $R_0$  is in the fixed part of $|2K_S|$. Then the number $\ell$ of disjoint $(-2)$-curves that form $R_0$ is  greater or equal than 2. In fact
\begin{equation*}
h_1+8=h_1+2h_2=6+\frac{3R_0K_S-R_0^2}2=6+\ell
\end{equation*}
forces $h_1=\ell-2$ and $\ell\ge2$. 

Let $M$ be an effective divisor in the movable part of the pencil $|2K_S-R_0|$. Then $M$ is in the movable part of the bicanonical system $|2K_S|=|M|+T$ and, by \cite{M}, either  $M^2=0$ or $M^2=2$. In any case the general curve of $|M|$ is smooth. 
From \cite[theorem 5.1]{CP} we can exclude the case $M^2=0$.

The strict transform $\widetilde M$ of $M$ satisfies $\widetilde M=\st\pi{M'}$ for some pencil $|M'|$ on $Y$.
This implies $\widetilde M^2\equiv 0 \mod 3$.
  Therefore $\widetilde M^2=0$.
We have $\st\varepsilon M=\widetilde M+D$ where $D$ is a sum of exceptional divisors.

Since $M^2=2$ then $D\ne 0$ and the general curve $M$, see lemma \ref{l.drop1}, passes either through one of the $h_2=4$ points $q_i$ (without loss of generality we may assume it is $q_1$) with multiplicities $m_1=1$,  $m_2=m_3=m_4=0$ or through two of the points $p_j$ (if $\ell\ge 4$). 
In the former case  $D=2F+G+H$ whereas in the latter case $D=E_1+E_2$.
In any case $p_a(\widetilde M)=p_a(M)=3$ and we have $\widetilde MK_X=4$.

\begin{lemma}\label{l.M'}
 $|M'|$ is a pencil of elliptic curves with ${M'}^2=0$.
\end{lemma}

\begin{proof}
We have 
\begin{align*}
3M'K_Y&=\st\pi{M'}\st\pi{K_Y}=\widetilde M(K_X-2\st\varepsilon{R_0}-2E-2F-2H)\\
&=\begin{cases}
4-2MR_0-2=2-2MR_0 \quad &\text{if}\  D=2F+G+H\\
4-2MR_0-4=-2MR_0 \quad &\text{if}\  D=E_1+E_2
\end{cases}
\end{align*}
and then $MR_0\equiv 0,1 \mod 3$. Since $0\le MR_0\le MT=2$ we get $MR_0=0,1$ hence 
$M'K_Y=0$. Since ${M'}^2=0$ this proves the lemma.
\end{proof}

\begin{theorem}\label{t.ii}
Case (ii) cannot occur.
\end{theorem}

\begin{proof}
Let us consider on $Y$ the fibration over $\mathbb P^1$ given by the elliptic pencil $|M'|$. From \cite[Proposition III.11.4]{BPV} we have
\begin{equation}\label{e.fibr1}
e(Y)=e(M')e(\mathbb P^1)+\sum_{s}(e(M'_s)-e(M'))=e(M')e(\mathbb P^1)+\delta
\end{equation}
where the sum runs over all the singular curves $M'_s$ in $|M'|$ and we set
$\delta:=\sum_{s}(e(M'_s)-e(M))$. 
Since $e(M')=0$ and from \eqref{e.ky} 
$e(Y)=12-K_Y^2=15+3\ell$ 
we find $\delta=15+3\ell$. 

 The general  $M$ on $S$ passes only through at most one of the points $q_j$. Then we have
$F'_2M'=F'_3M'=F'_4M'=H'_2M'=H'_3M'=H'_4M'=0$
and each of these disjoint curves contributes  3  to $\delta$ by lemma \ref{l.-n}. Moreover since $0\le MR_0\le 1$ (see the above proof) we have at least $\ell-1$ irreducible components $B_{0i}$ of $B_0$  not intersecting $M'$. Each curve $B_{0i}$ contributes  6 more nodes to $\delta$. Therefore
\begin{equation*}
6\cdot 3+6(\ell-1)=12+6\ell\le \delta=15+3\ell
\end{equation*}
which forces $\ell\le 1$ and we get a contradiction since we know $\ell\ge 2$.
\end{proof}

\section{Case (iii): $R_0K_S=0$, $h_2=1$}\label{s.iii}

In this case, from formula \eqref{e.fixed}, $h_1=4+\ell$ where $\ell$ is the number of irreducible components of $R_0$. Since $ h^0(Y,\mathcal O_Y(2K_Y+B))=0$, from Castelnuovo's theorem (see for example \cite{BPV}) it is immediate to see
\begin{prop}\label{p.rat}
If case (iii) holds then $Y$ is a rational surface.
\end{prop}

We still have the pencil $|N|=|A'|+\Phi'$  which is composed of curves of arithmetic genus 3.
Off the ${A'}^2$ base points $\phi_{|A'|}$ is a fibration over $\mathbb P^1$ of curves of genus $0\le p_a(A')\le 2$. Computing Euler numbers we obtain
\begin{equation}\label{e.fibr}
e(Y)+{A'}^2=e(A')e(\mathbb P^1)+\sum_{s}(e(A'_s)-e(A'))=e(A')e(\mathbb P^1)+\delta
\end{equation}
where the sum is taken over all the singular curves $A'_s$ in $|A'|$ and we set, as before,
$\delta:=\sum_{s}(e(A'_s)-e(A'))$. 
From lemma \ref{l.-n} we have

\begin{lemma}\label{l.nodes}
In the above setting each of the exceptional curves $E'_k$, $F'$ and $H'$ which does not intersect $A'$ increases $\delta$ by 3. Moreover each component $B_{0i}$ of $B_0$ for which $B_{0i}A'=0$ increases $\delta$ by 6.
\end{lemma}

\begin{lemma}\label{l.ab} 
In the above setting  we have
$\delta=14+3\ell+3{A'}^2+2A'K_Y$.
\end{lemma}

\begin{proof}
Let us compute, using \eqref{e.KY2}, 
$e(Y)+{A'}^2=12-K_Y^2+{A'}^2=14+{A'}^2+3\ell$
while
$e(A')e(\mathbb P^1)=(2-2\cdot p_a(A'))2=2(-{A'}^2-A'K_Y)$.
Therefore we have
\begin{equation*}
\delta=\sum_{s}(e(A'_s)-e(A'))=14+3\ell+3{A'}^2+2A'K_Y
\end{equation*}
as wanted.
\end{proof}

\begin{prop}\label{p.0}
Assume ${A'}^2=0$. Then $0\le\ell\le 1$ and we have $\ell=0$ only when $(0a)$, $(0c)$, $(0f)$ or $(0g)$ holds and $\ell=1$ only when $(0d)$  holds. Moreover cases $(0b)$, $(0e)$ and $(0h)$  of the list of proposition \ref{p.list0} cannot occur.
\end{prop}

\begin{proof}
We refer to the list of proposition \ref{p.list0}. From lemma \ref{l.ab} we find 
\begin{equation*}
\delta=14+3\ell+2A'K_Y
\end{equation*}

We have $A'K_Y=0$ in all cases of the list except for $(0g)$.

In case $(0a)$ we have to consider $F',H'$ and $h_1-2=2+\ell$ curves $E'_k$, plus all the components of $B_0$. Then
\begin{equation*}
6+3(2+\ell)+6\ell=12+9\ell\le\delta=14+3\ell
\end{equation*}
which forces $\ell=0$.

In case $(0b)$ we find $\ell\ge 1$ and we have the contribution of the curves $E'_k$, $H'$ (see corollary \ref{c.AH}) and of $\ell-1$ components of $B_0$. Then
\begin{equation*}
3(4+\ell)+3+6(\ell-1)=9+9\ell\le\delta=14+3\ell
\end{equation*}
which implies $6\ell\le 5$. Impossible.

In case $(0c)$  we have the contribution of $F',H'$ and $h_1-3=1+\ell$ of the curves $E'_k$, plus all the components of $B_0$. Then
\begin{equation*}
6+3(1+\ell)+6\ell=9+9\ell\le\delta=14+3\ell
\end{equation*}
which forces $\ell=0$.
The rest of the proof goes similarly.
\end{proof}

A similar argument to proposition \ref{p.0} shows (see also \cite[prop. 3.3.6, 3.3.7]{Pa})
\begin{prop}\label{p.1}
Assume ${A'}^2=1$. Then $0\le\ell\le 3$. 
If $(1f)$ holds then $\ell=0,1$. If $(1e)$ holds then $\ell=1,2,3$. If  $(1a)$ or $(1d)$ holds then $\ell=0$. 
Cases $(1b)$ and $(1c)$ of the list of proposition \ref{p.list1} cannot occur.
Moreover when case $(1e)$ holds for any irreducible component $B_{0k}$ of $B_0$ we find $B_{0k}A'\ge 1$.
\end{prop}

From proposition \ref{p.list2} it cannot be ${A'}^2=2$. Then (see remark \ref{r.N}) we are left with the case $A'=N$.
\begin{prop}\label{p.3}
Assume $A'=N$. Then $0\le\ell\le 1$. 
\end{prop}

As a consequence of propositions \ref{p.list0}, \ref{p.list1}, \ref{p.list2}, \ref{p.0}, \ref{p.1}, \ref{p.3} and of remark \ref{r.N} we obtain
\begin{theorem}\label{t.iii}
Case (iii) of page \pageref{le.sub} can only occur  
when one of the following conditions is satisfied:
\begin{enumerate}
\item $\ell=0$: Cases $(0a)$, $(0c)$, $(0f)$, $(0g)$, $(1a)$, $(1d)$, $(1f)$, $A'=N$ 
\item $\ell=1$: Cases  $(0d)$, $(1e)$, $(1f)$, $A'=N$ 
\item $\ell=2,3$:  Case $(1e)$ 
\end{enumerate}

Moreover in cases $(0g)$, $(1f)$ and $A'=N$ we have $\Phi=0$, i.e. the invariant pencil $\Lambda\le|3K_S|$ has no fixed part.
\end{theorem}

\section{More on the case $R_0K_S=0$, $h_2=1$}\label{s.more}

In the above setting from \eqref{e.fixed} $h_1=4+\ell$ and from \eqref{e.KY2} we have 
\begin{equation}\label{e.k2}
K_Y^2=\frac 1 3[K_S ^2-6-h_2+\frac 9 2R_0 ^2-\frac{11} 2R_0K_S]=
\frac 1 3[1-6-1-9\ell]=-2-3\ell
\end{equation}
From now on we refer to the formulas of proposition \ref{p.comp} when computing the arithmetic data of the curves in the linear systems $|N|$, $|N_1|$, $|N_2|$, $|N_3|$.

We start by computing $N_1^2$ and $p_a(N_1)$:
\begin{equation}\label{e.N12}
N_1^2=5-4R_0K_S+K_Y^2+n+h_2=5-2-3\ell+n+1=4-3\ell+n\ge 0
\end{equation}
\begin{equation}\label{e.paN1}
p_a(N_1)=4-3R_0K_S+K_Y^2+n+h_2=4-2-3\ell+n+1=N_1^2-1
\end{equation}

We have the following 
\begin{lemma}\label{l.n}
In the above setting we have $3\ell-4\le n\le 3\ell$.
\end{lemma}

\begin{proof}
Let us consider the short exact sequence of sheaves
\begin{equation}\label{e.n}
0 \longrightarrow \mathcal O_Y(N-N_1) \longrightarrow \mathcal O_Y(N) \longrightarrow \mathcal O_{N_1}(N) \longrightarrow 0
\end{equation}
Then, since $Y$ is a rational surface, from the definition of $N_1$,
\begin{align*}
&h^2(Y,\mathcal O_Y(N-N_1))=h^0(Y,\mathcal O_Y(K_Y-N+N_1))\\
&=h^0(Y,\mathcal O_Y(2K_Y-G'-\sum_{i=1}^n Z_i))
\le h^0(Y,\mathcal O_Y(2K_Y))=0
\end{align*}
The divisor $N-N_1$ cannot be effective, otherwise
$3=h^0(Y,\mathcal O_Y(N_1))\le h^0(Y,\mathcal O_Y(N))=2$.
The long exact sequence of \eqref{e.n} yields therefore
\begin{equation*}
0 \to H^0(Y,\mathcal O_Y(N)) \longrightarrow H^0(N_1,\mathcal O_{N_1}(N)) \longrightarrow H^1(Y,\mathcal O_Y(N-N_1)) \to 0
\end{equation*}
\begin{equation*}
0 \longrightarrow H^1(N_1,\mathcal O_{N_1}(N)) \longrightarrow 0 
\end{equation*}
This forces $H^1(N_1,\mathcal O_{N_1}(N))=0$. Since $N_1$ is big and nef, hence 1-connected,
\begin{equation*}
h^0(N_1,\mathcal O_{N_1}(N))=\chi(\mathcal O_{N_1}(N))=1+NN_1-p_a(N_1)
=2+3\ell-n
\end{equation*}
Then 
$2=h^0(Y,\mathcal O_Y(N))\le h^0(N_1,\mathcal O_{N_1}(N))=2+3\ell-n$
and $3\ell-4\le n\le 3\ell$  as wanted (see formula \eqref{e.N12}).
\end{proof}

\begin{remark}\label{r.0}
When $\ell=0$ one can easily see that $n=3\ell=0$ is the only possibility for $n$.
\end{remark}

We will see in the following sections that a deeper study of the adjoint linear systems $|N_i|$ to the pencil $|N|$ on $Y$ allows us to collect the cases listed in theorem \ref{t.iii} into two main groups 

\begin{definition}\label{d.rul}
We call {\bf ruled cases} those for which one of the linear systems $|N_i|$ induce a  morphism $Y\longrightarrow\mathbb F_a$ for some $a\ge 0$.
\end{definition}

\begin{definition}\label{d.DP}
We  call {\bf Del Pezzo cases} those which are not ruled cases.
\end{definition}

In section \ref{s.other} we will show that not all the cases listed in theorem \ref{t.iii} can actually occur.

\subsection{$n=3\ell-4$}

\begin{prop}\label{p.no4}
In the case $n=3\ell-4$ the net $|N_1|$ has no fixed part and we have $|N_1|=|2\Theta|$ where $|\Theta|$ is a pencil of rational curves with $\Theta^2=0$.
\end{prop}

\begin{proof}
Assume $n=3\ell-4$. Then $|N_1|$ is a net of curves with $N_1^2=0$. Let us write $|N_1|=|\Delta|+T$ where $T$  and $|\Delta|$ are the fixed and  the movable part of $|N_1|$ respectively. Then $0\le \Delta^2\le\Delta N_1\le N_1^2$ hence $\Delta N_1=0$ and
\begin{eqnarray*}
&0=N_1\Delta=\Delta^2+T\Delta\\
&0=N_1T=\Delta T+T^2
\end{eqnarray*}
It follows $\Delta^2=\Delta T=T^2=0$. Therefore there exists a pencil $|\Theta|$ such that $|\Delta|=|2\Theta|$. Then
$4=NN_1=2N\Theta+NT$
and $N\Theta\ge 1$ otherwise by the Index theorem and the rationality of $Y$ we have $\Theta\equiv 0$.

If $N\Theta=1$, $NT=N\Delta=2$ then 
$(\Delta-T)^2=\Delta^2+T^2-2\Delta T=0$
hence $\Delta\equiv T$ which is impossible.  Thus $N\Theta=2$ and $NT=0$ which forces $T=0$  by the Index theorem.
Therefore we have $N_1=\Delta$ and
\begin{equation*}
-1=p_a(N_1)=p_a(\Delta)=p_a(2\Theta)=1+\frac{2\Theta(2\Theta+K_Y)}2=1+\Theta K_Y
\end{equation*}
forces $\Theta K_Y=-2$. Then $|\Theta|$ is a pencil of rational curves. 
\end{proof}

\begin{theorem}\label{t.no4}
The case $n=3\ell-4$ cannot occur.
\end{theorem}

\begin{proof}
If $n=3\ell-4\ge 0$ we have  $\ell\ge 2$ and case $(1e)$ of proposition \ref{p.list1} holds (see also theorem \ref{t.iii}).
Therefore we have a pencil of elliptic curves $|A'|$ for which ${A'}^2=1$.
Then, from propositions \ref{pr.fix} and \ref{p.no4}, lemma \ref{le.sub}  and corollary \ref{c.AG}, 
$2A'\Theta=A'N_1=2-A'\sum_{i=1}^n Z_i\le 2$
and $A'N_1\ge 1$ by the Index theorem, whence $A'N_1=2$ and $A'\Theta=1$.

We have $ h^0(A',\mathcal O_{A'}(\Theta))=2$ since otherwise 
 the point $A'\cap \Theta$ should be a base point for the pencil $|\Theta|$, whereas $\Theta^2=0$. Then we get a contradiction since for any divisor $D$ of degree 1 on the smooth elliptic curve $A'$  $h^0(A',\mathcal O_{A'}(D))=1$. 
\end{proof}

\subsection{$n=3\ell-3$}\label{s.3l-3}

In this case we have $1\le\ell\le 3$ and from equations \eqref{e.N12}  and \eqref{e.paN1} we find $N_1^2=1$ and $p_a(N_1)=0$.

\begin{lemma}
If $n=3\ell-3$ then $|N_1|$ has no fixed part. Then the general element of $|N_1|$ is a smooth rational curve.
\end{lemma}
 
\begin{proof}
We can use the same argument as in  lemma \ref{l.N1} and we find that $|N_1|$ has no fixed part unless $|N_1|=|\Delta|+T$ with ${\Delta}^2=0, \Delta N_1=\Delta T=1$, $T^2=-1$. Since $\Delta^2=0$ and $|\Delta|$ is a net, there exists a pencil $|\Theta|$ such that $\Delta\equiv 2\Theta$. But then
$1=\Delta T=2\Theta T$
and we get a contradiction.
\end{proof}

In this setting $|N_1|$ 
is base point free and  $\phi_{|N_1|}:Y\longrightarrow \mathbb P^2$  is a birational morphism. 
\vskip .3 cm

When $n\ge 3\ell-2$ it makes sense to consider $|N_2|=|N_1+K_Y-\sum_{i=1}^n Z_i-G'-\sum_{j=1}^{n'} Z'_j|$ which is a linear system of dimension $3-3\ell+n=p_a(N_1)\le 3$ and $3-3\ell+n\ge 1$ and from proposition \ref{p.comp}
\begin{align}\label{e.N22}
N_2^2&=7-8R_0K_S+4K_Y^2+4n+4h_2+n'
=3-12\ell+4n+n'\notag\\
p_a(N_2)&=5-5R_0K_S+3K_Y^2+3n+3h_2+n'
=2-9\ell+3n+n'\\
N_1N_2&=2p_a(N_1)-2=2(3-3\ell+n)-2=4-6\ell+2n\notag
\end{align}

\subsection{$n=3\ell-2$}\label{s.3l-2}

In this case $1\le\ell\le 3$ and we have $N_1N_2=0$ and by the Index theorem  we infer $N_2\equiv 0$. Then from \eqref{e.N22} $n'=N_2^2-3+12\ell-4n=5$ and 
\begin{eqnarray}\label{e.3l-2}
&N_1\equiv \sum_{i=1}^nZ_i+G'+\sum_{j=1}^{5} Z'_j-K_Y\notag\\
&N\equiv 2\sum_{i=1}^nZ_i+2G'+\sum_{j=1}^{5} Z'_j-2K_Y\\
&2B_0+E'\equiv 2\sum_{i=1}^nZ_i+5G'+\sum_{j=1}^{5} Z'_j-5K_Y\notag
\end{eqnarray}

\subsection{$n=3\ell-1$}\label{s.3l-1}

This case can occur for $1\le\ell\le 3$. We have 
 $N_1N_2=2$, $N_1^2=3$ (cf. equations \eqref{e.N12}, \eqref{e.paN1} and \eqref{e.N22}) and since $N_1(3N_2-2N_1)=0$ then
\begin{equation*}
(3N_2-2N_1)^2=9N_2^2+4N_1^2-12N_1N_2=9N_2^2+12-24=9N_2^2-12\le 0
\end{equation*}
Hence from \eqref{e.N22} $0\le N_2^2=n'-1\le 1$ and $n'=1,2$.

\begin{lemma}\label{l.nofix}
If $N_2^2=0$ (i.e. $n'=1$) then $|N_2|$ has no fixed part. In particular the general member of $|N_2|$ is a smooth rational curve with self-intersection 0.
\end{lemma}

\begin{proof}
Let us write $|N_2|=|\Delta|+T$ with $|\Delta|$ and $T$ the movable and the fixed part of $N_2$ respectively.
We have $0=N_2^2=N_2\Delta+N_2T$ or, equivalently,
\begin{eqnarray*}
0=N_2\Delta={\Delta}^2+\Delta T\\
0=N_2T=\Delta T+T^2
\end{eqnarray*}
which implies ${\Delta}^2=\Delta T=T^2=0$.

We know that $N_1N_2=2$. Then $0\le N_1\Delta\le N_1N_2=2$. It cannot be $N_1\Delta=0$ otherwise, by the index theorem and the rationality of $Y$, $\Delta=0$.

It cannot  be $N_1\Delta=1=N_1T$: we obtain by the Index theorem
\begin{equation*}
0\ge(\Delta-T)^2={\Delta}^2+T^2-2\Delta T=0 \quad \Rightarrow \quad \Delta\sim T.
\end{equation*}
Since $Y$ is a rational surface, this implies $\Delta\equiv T$ which is impossible.

So $N_1\Delta=N_1N_2=2$ and then
$0\ge (N_2-\Delta)^2=T^2=0$.
Again, by the rationality of $Y$ we have $T\equiv 0$ and $|N_2|$ has no fixed part.
\end{proof}

 Then if $n'=1$ there exists a morphism $Y\longrightarrow \mathbb F_a$ for some $a\ge 0$.

If $n'=2$ then $|N_2|$ is a pencil of curves with arithmetic genus 1 and therefore $N_3\equiv 0$.
But now from proposition \ref{p.comp}
\begin{equation*}
N_3^2=9-12R_0K_S +9K_Y^2+9h_2+9n+4n'+n''
=n''-1=0
\end{equation*}
Then, recalling the definition of $N:=3K_Y+2B_0+E'-3G'$,
\begin{eqnarray*}
&N_2\equiv G'+\sum_{i=1}^nZ_i+\sum_{i=1}^2Z'_i+Z''-K_Y\\
&N_1\equiv 2G'+2\sum_{i=1}^nZ_i+2\sum_{i=1}^2Z'_i+Z''-2K_Y\\
&N\equiv 3G'+3\sum_{i=1}^nZ_i+2\sum_{i=1}^2Z'_i+Z''-3K_Y\\
&2B_0+E'
\equiv 6G'+3\sum_{i=1}^nZ_i+2\sum_{i=1}^2Z'_i+Z''-6K_Y\\
\end{eqnarray*}

\subsection{$n=3\ell$}\label{s.3l}

In this case we have $0\le\ell\le 3$. Now $N_1^2=4=N_1N_2$ and
$N_2^2=n'+3\ge 3$.
By the Index theorem
$(N_1-N_2)^2=n'-1\le 0$.
Moreover if $n'=1$ we have $N_1\equiv N_2$ and then
$K_Y\equiv G'+\sum_{i=1}^nZ_i+Z'$
which is impossible since $K_Y$ is not effective. This implies $n'=0$, $N_2^2=3$ and $p_a(N_2)=2$.

If we look at $N_3$ we have (see also proposition \ref{p.comp})
\begin{equation*}
N_3^2=9-12R_0K_S +9K_Y^2+9h_2+9n+4n'+n''
=n'' 
\end{equation*}
\begin{equation*}
p_a(N_3)=6-7R_0K_S+6K_Y^2+6h_2+6n+3n'+n''
=n''
\end{equation*}
Since $N_2N_3=2p_a(N_2)-2=2$ we have
$(3N_3-2N_2)^2=9n''-12\le 0$
hence $n''=0,1$. In the former case $|N_3|$ is a pencil of rational curves of self-intersection 0 (see also proposition \ref{p.comp}), whereas in the latter case we have a pencil of curves with arithmetic genus one.
 Again we infer 
\begin{equation*}
N_4=N_3+K_Y-G'-\sum_{i=1}^{n}Z_i-Z''-\sum_{i=1}^{n'''}Z_i'''\equiv 0.
\end{equation*}
Then 
$N_4^2=N_3^2+K_Y^2+2N_3K_Y+1+n+n'+n''+n'''=n'''-1=0$. Therefore
\begin{eqnarray*}
&N_3\equiv G'+\sum_{i=1}^{n}Z_i+Z''+Z'''-K_Y\\
&N_2\equiv 2G'+2\sum_{i=1}^{n}Z_i+2Z''+Z'''-2K_Y\\
&N_1\equiv 3G'+3\sum_{i=1}^{n}Z_i+2Z''+Z'''-3K_Y\\
&N\equiv 4G'+4\sum_{i=1}^{n}Z_i+2Z''+Z'''-4K_Y\\
&2B_0+E'\equiv 7G'+4\sum_{i=1}^{n}Z_i+2Z''+Z'''-7K_Y\\
\end{eqnarray*}

In  case $|N_3|$ is a pencil of rational curves  we can show arguing as in lemma \ref{l.nofix} that $|N_3|$ has no fixed part. 
Therefore we have a map $Y\longrightarrow \mathbb F_a$ for some $a\ge 0$.

\section{Further results}\label{s.other}
\begin{prop}\label{p.1e}
Case $(1e)$ of proposition \ref{p.list1} cannot occur.
\end{prop}

\begin{proof}
Assume case $(1e)$ holds. Then $Y$ has an elliptic pencil $|A'|$ with ${A'}^2=1$. From \cite[lemma 2.2]{CCM2}  if $A'+K_Y$ is not nef then there exist $(-1)$-cycles $D_j$ such that $D_jA'=0$.  By the Index theorem   we have 
\begin{equation}\label{e.a10}
A_1=A'+K_Y-\sum_{j}D_j\equiv 0
\end{equation}

We now look at the intersection number $s:=A'N_1$: we have (cf. corollary \ref{c.AG})
$A'N_1=A'(N+K_Y-G'-\sum_{i=1}^n Z_i)=2-A'\sum_{i=1}^n Z_i\ge 1$
otherwise ${N_1}\equiv 0$.  Moreover from theorem \ref{t.no4} we have $1\le N_1^2\le 4$ (cf. sections \ref{s.3l-3}, \ref{s.3l-2}, \ref{s.3l-1}  and  \ref{s.3l}).
Then by the Index theorem   we have $A'(N_1-sA')=0$ and
$(N_1-sA')^2=N_1^2-s^2\le 0$. 

If it was  $s=1$ we should have $N_1^2=1$ and, from the rationality of $Y$, $N_1\equiv A'$ which is impossible since $|A'|$ is a pencil whereas $|N_1|$ is a net.
Hence $A'N_1=2$ and $A'\sum_{i=1}^n Z_i=0$.
Moreover, when $N_1^2=4$ (or, equivalently, $n=3\ell$) we get, because of  the rationality of $Y$, 
$N_1\equiv 2A'$
which is impossible since then (see lemma \ref{le.N} and proposition \ref{p.list1}) 
$4=NN_1=2NA'=6$.

Hence in case $(1e)$ we can only have $n=3\ell-1$, $n=3\ell-2$ or $n=3\ell-3$.

Since $A'N_1=2$,  none of the curves $Z_i$ intersects $A'$. Therefore from \eqref{e.a10}
\begin{equation*}
A_1=A'+K_Y-G'-\sum_{i=1}^n Z_i-\sum_{j=1}^m C_j\equiv 0
\end{equation*}
and
\begin{equation}\label{e.a12}
0=A_1^2={A'}^2+K_Y^2+2A'K_Y+1+n+m
=n-2-3\ell+m
\end{equation}
hence $m=3\ell-n+2$.

We also know that in case $(1e)$ for any irreducible component $B_{0k}$ of $B_0$ we have  $B_{0k}A'\ge 1$ (see  proposition \ref{p.1}). Moreover $A'N=3, \Phi'N=0$. Then
\begin{equation*}
0=B_{0k}N=B_{0k}(A'+\Phi')\ge 1+B_{0k}\Phi'
\end{equation*}
for any $k=1,\dots,\ell$ forces $B_{0}\le \Phi'$. 
We recall that 
$N_1K_Y=(N+K_Y-G'-\sum_{i=1}^n Z_i)K_Y=n-3\ell$.
Then we find
\begin{align*}
0&=A_1N_1=N_1(A'+K_Y-G'-\sum_{i=1}^n Z_i-\sum_{j=1}^m C_j)\\
&=2+n-3\ell-N_1\sum_{j=1}^m C_j\le n-(3\ell-2)
\end{align*}

This excludes $n=3\ell-3$.

When $n=3\ell-2$ none of the $m=3\ell-n+2=4$ curves $C_j$ intersects $N_1$. Hence they are 4 of the $n'=5$ curves $Z'_i$. Since $N_2\equiv A_1\equiv 0$ we find
\begin{equation*}
N_1\equiv \sum_{i=1}^n Z_i+G'+\sum_{j=1}^{5} Z'_j-K_Y\equiv A'+Z'_5.
\end{equation*} 
Thus we get a contradiction since $0=F'N_1=F'A'+F'Z'_5=1$.

When $n=3\ell-1$ we have  $m=3\ell-n+2=3$ and, using proposition \ref{p.comp},
\begin{equation*}
0=A_1N_1=N_1(A'+K_Y-G'-\sum_{i=1}^n Z_i-\sum_{j=1}^m C_j)
=1-N_1\sum_{j=1}^m C_j.
\end{equation*}
Thus there is exactly one curve $C_1$ with $C_1N_1=1$ whereas the remaining two have to be chosen among the $n'\le 2$ curves $Z'_i$. This also excludes the case $n'=1$.

When $n'=2$ we have $A'N_2=A'(N_1+K_Y-G'-\sum_{i=1}^n Z_i-\sum_{j=1}^{n'}Z'_j)=1$
and $(A'-N_2)^2=0$.
Thus $A'\equiv N_2$ but we get a contradiction since
\begin{equation*}
\Phi'N_2=\Phi'A'=(N-A')A'=3-1=2
\end{equation*}
and $2\ge B_{0}N_2=B_{0}A'=3$.
\end{proof}

\begin{prop}\label{p.no0d}
The case $(0d)$ of proposition \ref{p.list0} cannot occur.
\end{prop}

\begin{proof}
Assume case $(0d)$ holds. Then we have $\ell=1$ (see theorem \ref{t.iii}) and $K_Y^2=-2-3\ell=-5$. Moreover $|A'|$ is a pencil of elliptic curves such that ${A'}^2=0=A'K_Y$. If we look at the adjoint system $A'+K_Y$ we find $A'(A'+K_Y)=0$. From \cite[lemma 2.2]{CCM2} there are $(-1)$-cycles $D_j$ such that
$A_1=A'+K_Y-\sum_{j=1}^mD_j$
is nef and $D_jA'=0$. Since $A'A_1=A'(A'+K_Y)=0$ we necessarily have $A_1^2=0$. Hence (recall that $G'A'=0$ from corollary \ref{c.AG})
\begin{equation*}
0=A_1^2=(A'+K_Y-G'-\sum_{j}C_j)^2=0+K_Y^2+0+1+m=-4+m
\end{equation*}
and 
\begin{equation*}
A_1K_Y=K_Y(A'+K_Y-G'-\sum_{j=1}^mC_j)=0-5+1+m=A_1^2=0.
\end{equation*}

Since from proposition \ref{pr.fix} $A'N=3$ we can write
\begin{equation*}
0\le A_1N=N(A'+K_Y-G'-\sum_{j=1}^mC_j)=3+1-N\sum_{j=1}^mC_j\le 4.
\end{equation*}

Assume now $1\le A_1N=s\le 4$. Then we have $N(3A_1-sA')=0$ and by the Index theorem   and the rationality of $Y$ $(3A_1-sA')^2=0$
and $sA'\equiv 3A_1$. 
Thus
$1\le s=sA'B_0=3A_1B_0$
forces $s=3$ and $A'\equiv A_1$ which is impossible since otherwise $K_Y$ would be effective. Thus $A_1N=0$ hence $A_1\equiv 0$ and 
\begin{equation}\label{e.a1}
0=A_1B_0=B_0(A'+K_Y-G'-\sum_{j=1}^4C_j)
=5-B_0\sum_{j=1}^4C_j. 
\end{equation}

We note that $B_0$ cannot be contained in any singular fibre of $|A'|$ since $B_0A'=1>0$. In particular it is not contained in any of the $(-1)$-cycles $C_j$. 
Then $B_0C_j\ge 0$ for any $j=1,\dots,4$ and from \eqref{e.a1} there exists a cycle $C_j$, say $C_1$ such that $C_1B_0\ge 2$.
But $C_1A'=0$ forces $C_1\le A'$ and
$2\le C_1B_0\le A'B_0=1$.  Contradiction.
\end{proof}

\begin{prop}\label{p.l0}
Case $(0a)$, $(0c)$ and $(0f)$ of proposition \ref{p.list0} cannot occur.
\end{prop}

\begin{proof}
Let us begin with case $(0a)$ of proposition \ref{p.list0}. Then $n=3\ell=n'=0$, $A'N=2$ and
$A'N_1=A'(N+K_Y-G')=A'N=2$
hence $\Phi'N_1=(N-A')N_1=2$.

 Moreover $A'N_2=A'(N_1+K_Y-G')=A'N_1=2$ and $NN_2=N(N_1+K_Y-G')=4+1=5$ which implies $\Phi'N_2=3$. But we know, from proposition \ref{p.list0},
$E'_i\Phi'=E'_i(N-A')=-1$
for $i=1,2$ which forces $E'_1+E'_2\le \Phi'$. Moreover
$E'_kN_2=E'_k(N+2K_Y-2G')=2$
for all $k=1,\dots,h_1$. 

Then we get a contradiction since $4=(E'_1+E'_2)N_2\le \Phi'N_2=3$.

Assume now that either  case $(0c)$ or case $(0f)$ of proposition \ref{p.list0} holds.  Then $n=3\ell=0$, $A'N=3, \Phi'N=0$ and $A'N_1=A'N=3$.
Then
$\Phi'N_1=(N-A')N_1=1$.
Since $N_1$ is nef there exists exactly one irreducible component $D$ of $\Phi'$ such that $DN_1=1$.
We also know, from proposition \ref{p.list0},
$E'_i\Phi'=E'_i(N-A')\le -1$
for $i=1,2$ and then since
$E'_kN_1=1$
for all $k=1,\dots,h_1$ we get a contradiction.
\end{proof}

From theorem \ref{t.iii} 
and propositions \ref{p.1e}, \ref{p.no0d} and \ref{p.l0} we obtain the following 
\begin{theorem}\label{t.iii2}
Case (iii) of page \pageref{le.sub} can only occur  
when one of the following conditions is satisfied:
\begin{enumerate}
\item $\ell=0$: Cases $(0g)$, $(1a)$, $(1d)$, $(1f)$, $A'=N$  
\item $\ell=1$:   Cases  $(1f)$,  $A'=N$ 
\end{enumerate}

Moreover in cases $(0g)$, $(1f)$ and $A'=N$ we have $\Phi=0$, i.e. the invariant pencil $\Lambda\le|3K_S|$ has no fixed part.
\end{theorem}

\section{Ruled cases}\label{s.rul}

\subsection{$n=3\ell$}

By definition \ref{d.rul} and the results of section \ref{s.3l} we know that $|\bar N_3|$ induces a morphism  $g:W \longrightarrow \mathbb F_a$ for some $a\ge 0$ with $K_W^2=K_Y^2+1+3\ell=-1$. Then we have 
\begin{equation*}
K_{\mathbb F_a}=-2c-(a+2)f, \quad \st g f=\bar N_3 
\end{equation*}
where $c$ is the $(-a)$-section of $\mathbb F_a$ and $f$ is a fibre of the ruling $\mathbb F_a\longrightarrow \mathbb P_1$. Then $K_W=-2\st g c-(a+2)\bar N_3+\Delta$  where $\Delta$ is the exceptional divisor of $g$.

Therefore
\begin{eqnarray*}
&\bar N_2=\bar N_3-K_W=2\st g c+(a+3)\bar N_3-\Delta\\
&\bar N_1=\bar N_2-K_W=4\st g c+(2a+5)\bar N_3-2\Delta\\
&\bar N=\bar N_1-K_W=6\st g c+(3a+7)\bar N_3-3\Delta
\end{eqnarray*}

\begin{lemma}\label{l.a2}
In the above setting $0\le a\le 2$.
\end{lemma}
\begin{proof}
Since $\bar N$ is nef we find 
$\st g c\bar N=7-3a\ge 0$ and then $0\le a\le 2$.
\end{proof}

We look at $2\bar B_0+\bar E'$: from the definition of $N$ on the surface $Y$ one has
\begin{equation}\label{e.rul}
2\bar B_0+\bar E'=\bar N-3K_W=12\st g c+(6a+13)\bar N_3-6\Delta
\end{equation}
whence $g(2\bar B_0+\bar E')=12c+(6a+13)f$. 
Furthermore 
\begin{equation*}
\st g{c}(2\bar B_0+\bar E')=\st g c(12\st g c+(6a+13)\bar N_3-6\Delta)=-12a+13+6a=13-6a
\end{equation*}

From theorem \ref{t.iii2} we have $\ell=0,1$.

\subsubsection{\underline{$\ell=0$}}
 
Let us assume $\ell=0$. Then $2\bar B_0+\bar E'=\bar E'$ and we can write 
 $\bar E'_i=\alpha_i\st g c+\beta_i\bar N_3+\sum_j \gamma_{ji}\Delta_j$. We also recall that $h_1=4+\ell=4$ from \eqref{e.fixed}. Then
\begin{equation*}
3=\bar N_3\bar E'_i=\bar N_3[\alpha_i\st g c+\beta_i\bar N_3+\sum_j \gamma_{ji}\Delta_j]=\alpha_i
\end{equation*}
This implies 
\begin{equation*}
\st g c\bar E'_i=\st g c[\alpha_i\st g c+\beta_i\bar N_3+\sum_j \gamma_{ji}\Delta_j]=-a\alpha_i+\beta_i=\beta_i-3a
\end{equation*}
\begin{lemma}\label{l.beta}
In the above setting we have $\st g c\bar E_i'\ge 0$. In particular we have  $\beta_i\ge 3a$ for all $i=1,\dots,4$.
\end{lemma}

\begin{proof}
It is obvious since $\st g c\bar E_i'<0$ and the irreducibility of $\bar E'_i$ would imply $\bar E'_i\le \st g c$ and therefore $\bar E'_i=\bar c$ the strict transform of $c$. 
Then we get a contradiction since 
$3=\bar E'_i\bar N_3=\bar c\bar N_3=\st g c\bar N_3=1$. 
\end{proof}

Moreover from equation \eqref{e.rul} we have $\sum_{i=1}^4 \beta_i=13+6a$.

\begin{lemma}\label{l.-1}
Each irreducible component $C$ in the singular fibres of $\phi_{|\bar N_3|}:W\longrightarrow \mathbb P^1$ is a rational curve  with $C^2=-1,-2$.  
\end{lemma}

\begin{proof}
Use  that $\bar N_3$ and $\bar N_2=\bar N_3-K_W$ are nef divisors and apply Zariski's lemma and the Index theorem.
\end{proof}

\begin{remark}\label{r.-2}
For any $(-2)$-curve $C'$ which is contained in a singular fibre we have 
$C'\bar E'=C'(\bar N_3-6K_W)=0$
so $C'$ does not intersect any of the curves $\bar E'_k$. Therefore the intersection of $\bar E'_k$ with the singular fibres is only given by the points of intersection with the $(-1)$-curves.
\end{remark}

\begin{lemma}\label{l.-1m}
Any singular fibre contains  two irreducible $(-1)$-curves with multiplicity 1. 
\end{lemma}

\begin{proof}
 Assume that there are $m$ irreducible $(-1)$-curves appearing with multiplicity $b_i\ge 1$ $i=1,\dots,m$. Then from lemma \ref{l.-1} and  the rationality of $|\bar N_3|$ we find
$-2=\bar N_3\bar K_W=-\sum_{i=1}^m b_i$
hence either $m=1, b_1=2$ or $m=2, b_1=b_2=1$.

Since $\bar E'_k\bar N_3=3$ for any $k=1,\dots,4$ and since the curves $\bar E'_k$ cannot intersect the $(-2)$-curves in each singular fibre (see remark \ref{r.-2}) there cannot be a fibre with only one $(-1)$-curve of multiplicity 2. 
\end{proof}

\begin{lemma}\label{l.2fibre}
For any singular fibre the curve $\bar E'$ intersects the exceptional curves of that fibre.
\end{lemma}

\begin{proof}
Assume there is a singular fibre $\psi$ of $g:W\longrightarrow \mathbb F_a$ such that $\bar E'$ does not intersect any of the exceptional curves of that fibre.
Then there exists a curve $\Gamma$ in $\psi$ such that $\Gamma\bar E'=\bar E'\bar N_3=12$ and $\Gamma$ is not contracted by $g:W\longrightarrow \mathbb F_a$. Hence
\begin{equation*}
12=\Gamma\bar E'=\Gamma(12\st g c+(13+6a)\bar N_3-6\Delta)=12\Gamma\st g c-6\Gamma\Delta.
\end{equation*}
Since $\Gamma\Delta\ge 1$ and $\Gamma\Delta\equiv 0 \mod 2$ we find
$12\le 12\Gamma \st g c-12$
hence $\Gamma\st g c\ge 2$. Let $g(\Gamma)=f_1$ be the  fibre of the ruling of $\mathbb F_a$ obtained by $\Gamma$. Then
$1=f_1c=\st g {f_1}\st g c=\Gamma\st g c\ge 2$
and we get  a contradiction.
\end{proof}

\begin{lemma}\label{l.a1}
In the above setting we can reduce to the case $a=1$ unless $a=2$ and $\phi_{|\bar N_3|}:W\longrightarrow \mathbb P^1$ has at most two singular fibres.
\end{lemma}

\begin{proof}
We know that $\bar F'\bar N_3=\bar H'\bar N_3=0$ and the two $(-2)$-curves are contained in a singular fibre  of $\phi_{|\bar N_3|}$. We can choose the map $g$ so that it contracts these curves to a point which is now on a nonsingular fibre $f_0$ of the map $\mathbb F_a\longrightarrow \mathbb P^1$. 

If $a=0$ and  we blow up the above point and we consider the section $c$ intersecting $f_0$ at that point, the strict transform of $c$ is a $(-1)$-curve. By contracting  the strict transform of $f_0$ the exceptional divisor becomes a curve with self-intersection 0. Therefore the surface now obtained is $\mathbb F_1$.

We can do the same for $a=2$ if the point  is not  the intersection point between $f_0$ and the $(-2)$-section $c$ on $\mathbb F_2$. 

Assume now  that $a=2$ and $c$ passes through the above point $P_0$. We can reduce to $a=1$ if we find a singular fibre $f_1$ such that $c$ does not pass through the point obtained by contracting all the exceptional curves of $g:W\longrightarrow \mathbb F_2$ in that fibre.

Let us suppose such a fibre $f_1$ does not exist. Then for any singular fibre $f'$ the $(-2)$-section $c$ passes through the point $P'$ which is the contraction of all the exceptional curves in that fibre.
From lemma \ref{l.2fibre} we can deduce that $P$ must be a point in  $\bar E'$. 
Since $\st g c \bar E'=13-6a=1$ there can be at most one such fibre. Thus, if the number of singular fibres is at least 3 we are done.
 \end{proof}

\begin{lemma}\label{l.neg}
For any $i=1,\dots,4$ and $j=1,\dots,9$ we have
\begin{eqnarray*}
&\Delta_k\sum_{j}\gamma_{ji}\Delta_j=0 \quad \text{if}\ \Delta_k^2=-2\\
&0\le \Delta_k\sum_{j}\gamma_{ji}\Delta_j\le 3 \quad \text{if}\ \Delta_k^2=-1.
\end{eqnarray*}
\end{lemma}

\begin{proof} 
Since $\bar E'_i=3\st g c+\beta_i\bar N_3+\sum_{j}\gamma_{ji}\Delta_j$ we have 
\begin{equation*}
0\le \bar E'_i\bar \Delta_k=\Delta_k\sum_j \gamma_{ji}\Delta_j\le \bar E'_i\bar N_3=3.
\end{equation*}
From lemma \ref{l.-1} the curves $\Delta_k$ have self-intersection $-2\le\Delta_k^2\le -1$. 

Furthermore, from remark \ref{r.-2}, if $\Delta_k^2=-2$ then 
$\bar E'_i\Delta_k=0$. 
Hence
$0=\bar E'_i\Delta_k=\Delta_k \sum_j\gamma_{ji}\Delta_j$
which proves the lemma.
\end{proof}

\begin{cor}\label{c.sum}
For any $i=1,\dots,4$ we have 
$\Delta\sum_{j}\gamma_{ji}\Delta_j\le 6r$
where $r$ is the number of singular fibres of $g:W\longrightarrow \mathbb F_a$.
\end{cor}

\begin{proof}
Let us set 
$V:=\{v\,|\,\Delta_v^2=-1\}$.
From the above proposition we have 
\begin{equation*}
\Delta\sum_{j}\gamma_{ji}\Delta_j=\sum_{v\in V}\Delta_v \gamma_{ji}\Delta_j\le 3|V|
\end{equation*}
where $|V|$ is the cardinality of the set $V$. From lemma \ref{l.-1m} there are two simple $(-1)$-curves in each of the $r$ singular fibres then $|V|\le 2r$ as wanted.
\end{proof}

We are now ready to show that the reduction to $a=1$ it is always possible. 

\begin{prop}\label{p.no2}
The case $a=2$ cannot occur with $r\le 2$ singular fibres.
\end{prop}

\begin{proof}
We know from the formulas of page \pageref{s.rul} that
$\bar N_2=2\st g c+(a+3)\bar N_3-\Delta$.
Then
\begin{align*}
2&=E'_iN_2=\bar E'_i\bar N_2=(3\st g c+\beta_i\bar N_3+\sum_{j}\gamma_{ji}\Delta_j)(2\st g c+(a+3)\bar N_3-\Delta)\\
&=-6a+3a+9+2\beta_i-\Delta\sum_j\gamma_{ji}\Delta_j
\end{align*}
hence
$\Delta\sum_j\gamma_{ji}\Delta_j=2\beta_i+7-3a$.
Thus, from lemma \ref{l.beta} and corollary \ref{c.sum},
\begin{equation*}
7+3a\le 2\beta_i+7-3a\le 6r
\end{equation*}
where $r$ is the number of singular fibres of $g:W\longrightarrow \mathbb F_a$. When 
$a=2$ we get $6r\ge 13$ and then $r\ge 3$ as wanted.
\end{proof}

From now on we assume $a=1$.
The pencil $|\bar N_3|$ is mapped to the pencil of lines of $\mathbb P^2$ through a  point $P$.
Then $|\bar N_2|$ maps to the net of quartics with 1 double point and 9 simple base points, $|\bar N_1|$ to the net of curves of degree 7 with 1 triple point and 9 double points (with no other simple base points), and
$|\bar N|$ to the pencil of curves of degree 10 with one quadruple point, 9 triple points and no other base points.

\begin{theorem}\label{t.no0}
The case  $n=3\ell=0,n'=n''=0$ cannot occur.
\end{theorem}

\begin{proof}
We compute the plane image of $|\bar A'|$. From theorem \ref{t.iii2} we know that $\ell=0$ can only occur in case $(0g)$ of proposition \ref{p.list0}, in cases $(1a)$, $(1d)$ or $(1f)$ of proposition \ref{p.list1} and when $A'=N$. Then, using also proposition \ref{pr.fix} and corollary \ref{c.AG}, ${A'}^2=0,A'K_Y=2, A'N=3, A'G'=1$ in the former case while we have  $A'K_Y=1, A'N=3, A'G'=0$ in the  latter cases with ${A'}^2=1$ unless $A'=N$. Hence we find
\begin{equation*}
A'N_3=A'(N+3K_Y-3G')=3+3A'K_Y-3A'G'=6
\end{equation*}
in all the above cases.
Then $A'$ is mapped onto a plane curve of degree $d$ with a point of multiplicity $d-6$ at $P$ and, denoting by $s_j$ the number of points of multiplicity $j$ among $P_1,\dots,P_9$,
\begin{equation*}
3=A'N=\bar A'\bar N=10d-3\sum_j js_j-4(d-6)
\end{equation*}
hence
\begin{equation}\label{e.ja}
\sum_j js_j=2d+7
\end{equation}

We also have $\bar{A'}^2=1, p_a(\bar A')=2$ in all the above cases except for $A'=N$. Then, if $A'\ne N$,
\begin{equation*}
1=\bar {A'}^2=d^2-\sum_j j^2s_j
\end{equation*}
hence 
\begin{equation}\label{e.j2a}
\sum_j j^2s_j=d^2-1
\end{equation}
and 
\begin{align*}
2=p_a(\bar A')&=\frac{(d-1)(d-2)}2-\sum_js_j\frac{j(j-1)}2-\frac{(d-6)(d-7)}2\\
&=\frac{d^2-3d+2-d^2+13d-42}2-\sum_js_j\frac{j(j-1)}2
\end{align*}
hence
\begin{equation}\label{e.jdiffa}
\sum_js_j j(j-1)=10d-44.
\end{equation}

Then comparing \eqref{e.ja}, \eqref{e.j2a} and \eqref{e.jdiffa} we get
\begin{equation*}
10d-44=d^2-1-(2d+7)=d^2-2d-8
\end{equation*}
hence
\begin{equation*}
d^2-12d+36=(d-6)^2=0
\end{equation*}
which forces $d=6$. In this case \eqref{e.ja} and \eqref{e.j2a} become
\begin{eqnarray*}
&\sum_j js_j=19\\
&\sum_j j^2s_j=35.
\end{eqnarray*}

We now easily infer $j\le 5$. Subtracting the first equation from the second one we find
\begin{align*}
16&=35-19=(25s_5+16s_4+9s_3+4s_2+s_1)+\\
&-(5s_5+4s_4+3s_3+2s_2+s_1)=20s_5+12s_4+6s_3+2s_2.
\end{align*}
Hence $s_5=0,s_4\le 1$.
Then we find 
$6s_3+2s_2=16-12s_4$
or equivalently, 
$3s_3+s_2=8-6s_4$
and substituting in \eqref{e.ja}
\begin{equation*}
s_1+s_2=19-4s_4-(3s_3+s_2)=19-4s_4-(8-6s_4)=11+2s_4
\end{equation*}
which gives a contradiction since $\sum_j s_j\le 9$.

We now discuss the case $A'=N$. Then $\bar {A'}^2=\bar N^2=3$ and $p_a(\bar A')=3$. Then
\begin{equation*}
3=\bar {A'}^2=d^2-\sum_j j^2s_j
\end{equation*}
forces 
\begin{equation}\label{e.j2n}
\sum_j j^2s_j=d^2-3
\end{equation}
and 
\begin{align*}
3=p_a(\bar A')&=\frac{(d-1)(d-2)}2-\sum_js_j\frac{j(j-1)}2-\frac{(d-6)(d-7)}2\\
&=\frac{d^2-3d+2-d^2+13d-42}2-\sum_js_j\frac{j(j-1)}2
\end{align*}
hence
\begin{equation}\label{e.jdiffn}
\sum_js_j j(j-1)=10d-46
\end{equation}

Thus comparing \eqref{e.ja}, \eqref{e.j2n} and \eqref{e.jdiffn} we find
\begin{equation*}
10d-46=d^2-3-(2d+7)=d^2-2d-10
\end{equation*}
hence
$d^2-12d+36=(d-6)^2=0$
forces $d=6$ while we know $d=10$ since $\bar A'=\bar N$.
\end{proof}

\subsubsection{\underline{$\ell=1$}}

Let us now assume $B_0\ne 0$.
For any curve $E'_i$ on the rational surface $Y$ there is at most one of the curves $Z_j$ intersecting $E'_i$ (see corollary \ref{c.Zint}). 

Assume $n=3\ell=3$. 
If all the cycles $Z_j$ are irreducible then there are exactly 3 of the 5 curves $E'_i$ (we can suppose they are $E'_{3},E'_{4}$ and $E'_5$)  intersected by one (and only one) of the curves $Z_j$: we have $E'_iN_3=0$ for each of them. Hence  they are contained in singular fibres of the map $\phi_{|N_3|}:Y\longrightarrow \mathbb P_1$. 

If one of the cycles $Z_j$ is reducible then, from corollary \ref{c.red}, either $Z_1,Z_2$ are irreducible and $Z_3=Z_1+Z_2+E'_k$ for some $1\le k\le h_1=5$ or  $Z_1$ is irreducible, $Z_2=Z_1+C$ (with $C$ a $(-2)$-curve) and  $Z_3=Z_1+Z_2+E'_k$ for some $k\le 5$. In any case we have $E'_kN_1=0$ (see proposition \ref{p.-1cycles}) hence $E'_kN_3=0$. Moreover since $\ell=1$  from lemma \ref{l.-3Z} and corollary \ref{c.Zint}
$B_0\sum_{i=1}^nZ_i=2B_0Z_1+2B_0Z_2=4$
hence $B_0N_1=0$ forces $B_0N_3=0$. 

We now can show the following
\begin{theorem}\label{t.no1rul}
The case $n=3\ell,n'=n''=0$ with $\ell=1$ cannot occur.
\end{theorem}

\begin{proof}
Let us now consider the fibration given by the rational pencil $|N_3|$. If we set $\delta:=\sum_s (e(N_{3s}-e(N_3))$ from \cite[Proposition III.11.4]{BPV} we have
\begin{equation*}
\delta=e(Y)-e(N_3)e(\mathbb P_1)=12-K_Y^2-4\overset{\eqref{e.k2}}{=}12+2+3\ell-4=13.
\end{equation*}

From lemma \ref{l.-n} every curve $C$ in a singular fibre contributes  $-C^2$ to $\delta$.
If all the cycles $Z_i$ are irreducible then  the $(-3)$-curves $F',H',E'_3,E'_{4},E'_{5}$ are disjoint and they are contained in singular fibres.

If $Z_3$ is reducible then the $(-3)$-curves $F',H',E'_k$ and the $(-6)$-curve $B_0$ are contained in singular fibres. 
In any case  we find a contradiction since $13=\delta \ge 15$.
\end{proof}

\subsection{$n=3\ell-1$}\label{s.rul1}

From section \ref{s.3l-1} we know that $|\bar N_2|$ gives a morphism  $g:W \longrightarrow \mathbb F_a$ for some $a\ge 0$ with $K_W^2=K_Y^2+1+3\ell-1+1=-1$. Then we have 
\begin{equation*}
K_{\mathbb F_a}=-2c-(a+2)f, \quad \st g f=\bar N_2 
\end{equation*}
and $K_W=-2\st g c-(a+2)\bar N_2+\Delta$  where $\Delta$ is the exceptional divisor of $g$.

Therefore
\begin{eqnarray*}
&\bar N_1=\bar N_2-K_W=2\st g c+(a+3)\bar N_2-\Delta\\
&\bar N=\bar N_1-K_W=4\st g c+(2a+5)\bar N_2-2\Delta\\
\end{eqnarray*}

Then similarly to the case $n=3\ell=0$ one can show (see also \cite[section 5.1.2]{Pa})
\begin{theorem}\label{t.no1}
The case $n=3\ell-1=2,n'=1$ cannot occur.
\end{theorem}

\section{Del Pezzo cases}

We now treat separately those cases with $\ell=0$ from those with $\ell=1$. We refer then to the list of theorem \ref{t.iii2}. Since the ideas of most proofs are quite similar we only write here two of them, i.e. we fix $\ell=1$ and we study the cases $n=3\ell$ and $n=3\ell-1$.

When $\ell=1$ from theorem \ref{t.iii2} we always have $\Phi=0$ and either case $(1f)$ of proposition \ref{p.list1} or $A'=N$ holds. Moreover 
we have $0=3\ell-3\le n\le 3\ell=3$.

\subsection{$n=3\ell$}

From the results of section \ref{s.3l} when we contract the $(-1)$-cycles $G', Z_1$, $Z_2$, $Z_3$, $Z''$, $Z'''$ we get a rational surface $W$ which is isomorphic to the projective plane $\mathbb P^2$ blown up at eight points $P_1,\dots, P_8$.
\vskip .3 cm

\underline{Case I}: The cycles $Z_1,Z_2,Z_3$ are irreducible.

In this case $B_0N_1=B_0(N+K_Y-G'-Z_1-Z_2-Z_3)=1$ and $B_0N_2=B_0(N+2K_Y-2G'-2\sum_{i=1}^3Z_i)=2$.

In particular, since $Z''N_2=0$ and $Z'''N_2=1$, from the nefness of $N_2$ one can see that $B_0$ cannot be an irreducible component of any of these cycles, i.e. $B_0$ is not contracted on $W$.
Let us compute  (recall that $N_4\equiv 0$)
\begin{equation*}
0=B_0N_4=B_0(N+4K_Y-4G'-4\sum_{i=1}^3Z_i-2Z''-Z''')
=4-B_0(2\sum_{i=1}^2Z'_i+Z'')
\end{equation*}
hence $0\le B_0Z''\le 2$ and we have
\begin{equation*}
\bar B_0\bar N_3=B_0N_3=B_0(N+3K_Y-3G'-3\sum_{i=1}^2Z_i-Z'')
=3-B_0Z''\ge 1.
\end{equation*}

Thus we can write the following table
\begin{equation*}
\begin{tabular}{c|cccccc}
&$B_0Z_1$  &$B_0Z_2$ &$B_0Z_3$  &$B_0Z''$ &$B_0Z'''$ &$\bar B_0\bar N_3$ \\
\hline
a) &1 &1 &1 &2 &0 &1 \\
b) &1 &1 &1 &1 &2 &2 \\
c) &1 &1 &1 &0 &4 &3 \\
\end{tabular}
\end{equation*}

We now apply the  Index theorem.  Since $\bar N_3^2=1$ 
and  $\bar B_0\bar N_3=s\ge 1$ we find $\bar B_0^2\le s^2$ which excludes case c) and forces $\bar B_0\equiv \bar N_3$ in case a).  We also note that in case b) we have $\bar B_0^2=2$.

\begin{lemma}\label{l.noa}
Case a) cannot occur.
\end{lemma}

\begin{proof}
Assume case a) holds. 
Since $\bar B_0\equiv \bar N_3$, $B_0Z_1=B_0Z_2=B_0Z_3=1$ and $E'_5Z_1=E'_4Z_2=1=E'_3Z_3$, if $E'_3,E'_4,E'_5$ were not contracted on $W$ we should have 
\begin{equation*}
1\le \bar E'_k\bar B_0=\bar E'_k\bar N_3=E'_kN_3=E'_k(N+3K_Y-3G'-3\sum_{i=1}^2Z_i-Z'')=-E'_kZ''
\end{equation*}
 for $k=3,4,5$. Then we should have $E'_k\le Z''$  and $0\le E'_kN_3\le Z''N_3=0$. Contradiction. 
Thus  $E'_k$ is contracted on $W$ and, from the nefness of $N_3$, $E'_kZ''=E'_kZ'''=0$ ($k=3,4,5$).
From the definition of $N$ and from $Z'''N=3$ we find
\begin{equation}\label{e.z4}
6=2B_0Z'''+E'Z'''=E'Z'''=(E'_1+E'_2)Z'''.
\end{equation}

Moreover, for  $k=1,2$ 
\begin{equation*}
0=E'_kN_4=E'_k(N+4K_Y-4G'-4\sum_{i=1}^2Z_i-2Z''-Z''')=4-E'_k(2Z''+Z''')
\end{equation*}
and
\begin{equation*}
E'_kN_3=E'_k(N+3K_Y-3G'-3\sum_{i=1}^2Z_i-Z'')
=3-E'_kZ''\ge 0
\end{equation*}

Since $E'_kN_2=2$ while $Z''N_2=0,Z'''N_2=1$, it cannot be $E'_k\le Z'', Z'''$ for $k=1,2$. In particular $E'_kZ'',E'_kZ'''\ge 0$.
Thus 
\begin{equation}\label{e.z5}
2E'_kZ''+E'_kZ'''=4
\end{equation}
 forces $0\le E'_kZ'''\le 4$
 and  from \eqref{e.z4} there should be at least one of the curves $E'_k$ (say $E'_2$) such that $E'_kZ'''=4, \bar E'_k\bar N_3=3$. We get a contradiction since from the Index theorem   we should have
$(\bar E'_2-3\bar N_2)^2\le 0$
or equivalently $\bar {E'}_2^2\le 9$.
\end{proof}

We now study case b). Let us denote by $d_0$ the degree of the plane image of $\bar B_0$. Since $2\bar B_0+\bar E'=\bar N-3K_W$ the curve $2\bar B_0+\bar E'$ is sent to an element of $|-7K_{\mathbb P^2}|$. 

We note that a quadratic transformation leaves the plane image of $|2\bar B_0+\bar E'|$ unchanged. In particular even after any quadratic transformation the equation
\begin{equation}\label{e.deg1a}
2d_0+\sum_{i=1}^5d_i=21
\end{equation}
holds, where $d_0$ is the degree of the image of $\bar B_0$ while $d_i,i=1,\dots,5$ are the degrees of the plane images of the curves $\bar E'_i$. In particular we have $d_0\le 10$.

The curve $\bar B_0$ satisfies the linear system 
\begin{equation}\label{e.sys}
\begin{cases}
\sum_j j^2s_j=d_0^2-\bar B_0^2=d_0^2-2\\
\sum_j js_j=3d_0-\bar B_0\bar N_2=3d_0-2\\
\sum_j s_j\le 8
\end{cases}
\end{equation}
where $s_j$ is the number of points among $P_1,\dots,P_8$ of multiplicity $j$ for $\bar B_0$.
By an easy computation one can see that $d_0\ge 3$ and we have the following list of solutions:
\begin{equation*}
\begin{array}{clcl}
1) & d_0=3, s_1=7  &5) & d_0=7, s_3=3, s_2=5 \\
2) & d_0=4, s_2=2, s_1=6 &6) & d_0=8, s_3=6, s_2=2 \\
3) &d_0=5, s_2=5, s_1=3   &7) &d_0=9, s_4=1, s_3=7 \\
4) & d_0=6, s_3=1, s_2=6, s_1=1 &\\
\end{array}
\end{equation*}

\begin{prop}\label{p.equiv0}
All the above solutions are equivalent up to a finite number of Cremona quadratic transformations of $\mathbb P^2$ based at $P_1,\dots,P_8$.
\end{prop}

\begin{proof}
Let us consider a curve of degree 9 as in 7) and let us take the quadruple point $Q_1$ and two of the seven triple points $Q_2,Q_3$. Then they are not collinear otherwise there should be a line meeting the above curve at 10 points.
Moreover $Q_1$ is on $\mathbb P^2$, i.e. it is not infinitely near to any other point, since is the unique point of maximal multiplicity for the curve.
Since $\bar E'_i$ is an irreducible curve if both $Q_2,Q_3$ were proximate to $P_1$  from the proximity inequalities (see \cite{Ca}) we should have 
$4=m_{P_1}\ge m_{P_2}+m_{P_3}=3+3=6$.

If $Q_2$ or $Q_3$ are  not infinitely near to $Q_1$ we can choose them to be on $\mathbb P^2$.
Then a quadratic transformation (see \cite{Ca}) based at $Q_1,Q_2,Q_3$ is well-defined and takes the curve of degree 9 onto an octic as in 6).

Let us consider the octic in 6) and let us take three of the six triple points $Q_1$, $Q_2$, $Q_3$. Then they are not collinear otherwise there should be a line meeting the octic at 9 points.
Moreover we can choose the points in such a way that one of them, say $Q_1$, is on $\mathbb P^2$, i.e. it is not infinitely near to any other point.
Since $\bar E'_i$ is an irreducible curve if both $Q_2,Q_3$ were proximate to $P_1$ from the proximity inequalities we should have 
$3\ge 3+3=6$.

If $Q_2$ or $Q_3$ are  not infinitely near to $Q_1$ we can choose them to be on $\mathbb P^2$.
Then a quadratic transformation  based at $Q_1,Q_2,Q_3$ is well-defined and takes the octic onto a septic as in 5).

With a similar argument one can see that we can choose two triple points and one double point for the septic such that there exists a quadratic transformation based at those points sending the septic to a sextic as in 4). To get 3) we consider the triple point $Q_1$ and two double points $Q_2,Q_3$ of the sextic. Then we can consider  three double points for the quintic such that a quadratic transformation based at those points sends the quintic onto a quartic as in 2). Eventually, if we base a quadratic transformation at the two double points of the quartic and at one of the six simple points, we can take the quartic onto the cubic in 1).
The result is then proved.
\end{proof}

From the above proposition, up to Cremona transformations, we can set  $d_0=9$. In particular we can assume that the quadruple point of the curve $B_0$ is $P_8$. Then we find 
\begin{equation*}
\sum_{i=1}^5d_i=21-2d_0=3
\end{equation*}
and $d_i\le 3$ for any $i=1,\dots,5$.

\begin{prop}\label{p.no3lirr}
In the above setting case I cannot occur.
\end{prop}

\begin{proof}
Let us consider the curves $E'_i$, $i=1,2$. Then
$E'_iN_2=E'_i(N+2K_Y-2G'-2\sum_{i=1}^3Z_i)=2$
while $Z''N_2=0,Z'''N_2=1$.
In particular, from the nefness of $N_2$, $E'_1$ and $E'_2$ cannot be contained  in $Z''$ or $Z'''$ and  they are not
contracted on $W$. Thus
\begin{equation}\label{e.12a}
E'_iZ_j=0,\quad  2E'_iZ''+E'_iZ'''=4 \quad (i=1,2, j=1,2,3).
\end{equation}

Since \eqref{e.12a} holds we have a priori three possibilities. In the former case $E'_iZ''=2, E'_iZ'''=0$ and then $\bar {E'}_i^2=1$ hence  $\bar E'_i\bar B_0=2$ and $d_i=3,s_1=8$. In the second case  $E'_iZ''=1, E'_iZ'''=2$ and then $\bar {E'}_i^2=2$ hence  $\bar E'_i\bar B_0=5$ and $d_i=3,s_1=7$.  In the latter case  $E'_iZ''=0, E'_iZ''=4$ hence $\bar E'_i\bar N_3=3$ and $\bar {E'}^2_i=13$ contradicting the Index theorem   as in the proof of lemma \ref{l.noa}.

Thus, since $\sum_{i=1}^5 d_i=3$, one among $E'_1$ and $E'_2$ is necessarily contracted on $W$ and we get a contradiction. 
\end{proof}
\vskip .3 cm

\underline{Case II}: At least one of the cycles $Z_i$ is reducible.

We know from corollary \ref{c.red} that either $Z_1,Z_2$ are irreducible and $Z_3=Z_1+Z_2+E'_k$ for a suitable $1\le k\le 5$  or $Z_1$ is irreducible $Z_2=Z_1+C$, $Z_3=2Z_1+C+E'_k$
for a suitable $1\le k\le 5$ where $C$ is a $(-2)$-curve.

Let us look at the $(-6)$-curve $B_0$. In any case we have $B_0Z_1=B_0Z_2=1$, $B_0Z_3=2$. 
\begin{prop}\label{p.3lred}
In the above setting case II cannot occur.
\end{prop}

\begin{proof}
If $B_0$ was contracted on $W$, then it should be contained either in $Z''$ or in $Z'''$. But  when we contract the cycles $Z_1,Z_2,Z_2$ the self-intersection of the image $B'_0$ of $B_0$ is
${B'}_0^2=B_0^2+1+1+4=0$.
Since $B_0$, hence $B'_0$, is irreducible, it cannot be a component of a $(-1)$-cycle. 

 Computing $B_0N_3$ and $B_0N_4$ and recalling that $N_3$ is nef whereas $N_4\equiv 0$, one can easily see that $B_0Z''=B_0Z'''=0$. Thus the image $\bar B_0$ of $B_0$ on the rational surface $W$ is a curve of self-intersection 0 having a node or a cusp (depending on the structure of the cycles $Z_i$) at the point obtained by contracting $Z_1,Z_2,Z_3$.

In particular we note that $\bar B_0\bar N_3=B_0N_3=0$. 
Hence by the Index theorem   and the rationality of $W$ we infer  $\bar B_0=0$. Contradiction.
\end{proof}

Hence from propositions \ref{p.no3lirr} and \ref{p.3lred} we obtain
\begin{theorem}\label{t.no3lDP1}
The case  $n=3\ell,n'=0,n''=n'''=1$ cannot occur with $\ell=1$.
\end{theorem}

\subsection{$n=3\ell-1$}

From the results of section \ref{s.3l-1} when we contract the $(-1)$-cycles $G', Z_1, Z_2, Z'_1$, $Z'_2, Z''$ we get a rational surface $W$ which is isomorphic to the projective plane $\mathbb P^2$ blown up at eight points $P_1,\dots, P_8$.

We also recall that from corollary \ref{c.red} the cycles $Z_1$ and $Z_2$ are irreducible $(-1)$-curves and
$B_0N_1=B_0(N+K_Y-G'-Z_1-Z_2)=2$. 
In particular, since $Z'_1N_1=Z'_2N_1=0$ and $Z''N_1=1$, from the nefness of $N_1$ one can see that $B_0$ cannot be an irreducible component of any of these cycles, i.e. $B_0$ is not contracted on $W$.
Let us compute  (recall that $0\equiv N_3\equiv N+3K_Y-3G'-3\sum_{i=1}^2Z_i-2\sum_{i=1}^2Z'_i-Z''$)
\begin{align*}
0=B_0N_3=6-B_0(2\sum_{i=1}^2Z'_i+Z'')
\end{align*}
hence $0\le B_0\sum_{i=1}^2Z'_i\le 3$ and we have
\begin{equation*}
\bar B_0\bar N_2=B_0N_2=B_0(N+2K_Y-2G'-2\sum_{i=1}^2Z_i-\sum_{i=1}^2Z'_i)
=4-B_0\sum_{i=1}^2Z'_i\ge 1.
\end{equation*}

Thus we can write the following table
\begin{equation*}
\begin{tabular}{c|cccccc}
&$B_0Z_1$  &$B_0Z_2$ &$B_0Z'_1$  &$B_0Z'_2$ &$B_0Z''$ &$\bar B_0\bar N_2$ \\
\hline
a) &1 &1 &3 &0 &0 &1 \\
b) &1 &1 &2 &1 &0 &1 \\
c) &1 &1 &1 &1 &2 &2 \\
d) &1 &1 &1 &0 &4 &3 \\
e) &1 &1 &0 &0 &6 &4 \\
\end{tabular}
\end{equation*}

We now apply the  Index theorem. Since $\bar N_2^2=1$ 
and  $\bar B_0\bar N_2=s\ge 1$ we find $\bar B_0^2\le s^2$ which excludes cases a), d), e) and forces $\bar B_0\equiv \bar N_2$ in case b). We also note that in case c) we have $\bar B_0^2=2$.

\begin{lemma}\label{l.nob}
Case b) cannot occur.
\end{lemma}

\begin{proof}
Assume case b) holds. 
Since $\bar B_0\equiv \bar N_2$, $B_0Z_1=B_0Z_2=1$ and $E'_5Z_1=E'_4Z_2=1$, if $E'_4,E'_5$ were not contracted on $W$ we should have 
\begin{equation*}
1\le \bar E'_k\bar B_0=E'_kN_2=E'_k(N+2K_Y-2G'-2\sum_{i=1}^2Z_i-\sum_{i=1}^2Z'_i)=-E'_k\sum_{i=1}^2Z'_i
\end{equation*}
 for $k=4,5$. Then we should have $E'_k\le Z'_i$ for some $i$ and $0\le E'_kN_2\le Z'_iN_2=0$. Contradiction. 

Thus $E'_k$ is contracted on $W$ and, from the nefness of $N_2$, 
$E'_k\sum_iZ''_i=E'_kZ''=0$ ($k=4,5$).
From the definition of $N$ and from $Z''N=2$ we find
\begin{equation}\label{e.z2}
5=2B_0Z''+E'Z''=E'Z''=(E'_1+E'_2+E'_3)Z''.
\end{equation}

Moreover, for  $k=1,2,3$, since $N_3=N+3K_Y-3G'-3\sum_{i=1}^2Z_i-2\sum_{i=1}^2Z'_i-Z''$, 
$0=E'_kN_3=3-E'_k(2\sum_{i=1}^2Z'_i+Z'')$
and
\begin{equation*}
E'_kN_2=E'_k(N+2K_Y-2G'-2\sum_{i=1}^2Z_i-\sum_{i=1}^2Z'_i)=2-E'_k\sum_{i=1}^2Z'_i\ge 0.
\end{equation*}

Thus 
\begin{equation}\label{e.z3}
2E'_k\sum_{i=1}^2Z'_i+E'_kZ''=3
\end{equation}
and $E'_k\sum_{i=1}^2Z'_i\le 2$ forces $-1\le E'_kZ''\le 3$.

Thus $E'_kZ''$ is odd and  from \eqref{e.z2} there should be at least one of the curves $E'_k$ (say $E'_3$) such that $E'_kZ''=3, \bar E'_k\bar N_2=2$. We get a contradiction since from the Index theorem   we should have
$(\bar E'_3-2\bar N_2)^2\le 0$
or equivalently $\bar {E'}_3^2\le 4$.
\end{proof}

We now study case c). Let us denote by $d_0$ the degree of the plane image of $\bar B_0$. Since $2\bar B_0+\bar E'=\bar N-3K_W$ the curve $2\bar B_0+\bar E'$ is sent to an element of $|-6K_{\mathbb P^2}|$. 

We note that a quadratic transformation leaves the plane image of $|2\bar B_0+\bar E'|$ unchanged. In particular even after any quadratic transformation the equation
\begin{equation}\label{e.deg1}
2d_0+\sum_{i=1}^5d_i=18
\end{equation}
holds, where $d_0$ is the degree of the image of $\bar B_0$ while $d_i,i=1,\dots,5$ are the degrees of the plane images of the curves $\bar E'_i$. In particular we have $d_0\le 9$.

The curve $\bar B_0$ satisfies the linear system 
\begin{equation*}
\begin{cases}
\sum_j j^2s_j=d_0^2-\bar B_0^2=d_0^2-2\\
\sum_j js_j=3d_0-\bar B_0\bar N_2=3d_0-2\\
\sum_j s_j\le 8
\end{cases}
\end{equation*}
where $s_j$ is the number of points among $P_1,\dots,P_8$ of multiplicity $j$ for $\bar B_0$.
Then, as for \eqref{e.sys}, $d_0\ge 3$ and we have the following list of solutions:
\begin{equation*}
\begin{array}{clcl}
1) & d_0=3, s_1=7  &5) & d_0=7, s_3=3, s_2=5 \\
2) & d_0=4, s_2=2, s_1=6 &6) & d_0=8, s_3=6, s_2=2 \\
3) &d_0=5, s_2=5, s_1=3   &7) &d_0=9, s_4=1, s_3=7 \\
4) & d_0=6, s_3=1, s_2=6, s_1=1 &\\
\end{array}
\end{equation*}

\begin{prop}
All the above solutions are equivalent up to a finite number of Cremona quadratic transformations of $\mathbb P^2$ based at $P_1,\dots,P_8$.
\end{prop}

\begin{proof}
See the proof of proposition \ref{p.equiv0}.
\end{proof}

From the above proposition, up to Cremona transformations, we can set  $d_0=8$. In particular we can assume that the two double points of the octic are $P_7,P_8$. Then we find 
\begin{equation*}
\sum_{i=1}^5d_i=18-2d_0=2
\end{equation*}
and $d_i\le 2$ for any $i=1,\dots,5$.
If one of the curves $E'_i$ is contracted on $W$ then it has $d_i=0$ and the multiplicity at each of the points $P_1,\dots,P_8$ is 0. 

If $E'_i$ is not contracted on $W$ we have two different numerical possibilities:
\begin{equation}\label{e.45}
E'_i\sum_{j=1}^2Z_j=1, E'_iZ'_1=E'_iZ'_2=0, E'_iZ''=0 \quad (i=4,5)
\end{equation}
\begin{equation}\label{e.123}
E'_iZ_1=E'_iZ_2=0,2E'_i\sum_{j=1}^2Z'_i+E'_iZ''=3 \quad (i=1,2,3).
\end{equation}

When \eqref{e.45} holds we find $\bar {E'}_i^2=-2$ hence, since $d_i\le 2$,  either $d_i=1,s_1=3$ or $d_i=2,s_1=6$. Moreover $\bar E'_i\bar B_0=1$.

When \eqref{e.123} holds we a priori have two possibilities. In the former case $E'_i\sum_{j=1}^2Z'_j=1, E'_iZ''=1$   and then $\bar {E'}_i^2=-1$ hence  $\bar E'_i\bar B_0=3$ and either $d_i=1,s_1=2$ or $d_i=2,s_1=5$. In the latter case  $E'_i\sum_{j=1}^2Z'_j=0, E'_iZ''=3$ hence $\bar E'_i\bar N_2=2$ and $\bar {E'}^2_i=6$ contradicting the Index theorem   as in the proof of lemma \ref{l.nob}.

We now fix $i=4,5$. Then since $\bar E'_i\bar N_2=0$ we have $\sum_{j=1}^8 m_j=3d_i$ where $m_j$ is the multiplicity of $\bar E'_i$ at $P_j$. Moreover
\begin{align*}
1&=\bar E'_i\bar B_0=8d_i-3(m_1+\dots+m_6)-2(m_7+m_8)\\
&=8d_i-3\sum_{j=1}^8 m_j+(m_7+m_8)=-d_i+m_7+m_8
\end{align*}
hence 
\begin{equation}\label{e.m45}
d_i+1=m_7+m_8\le 2
\end{equation}
since there are no singular points among $P_1,\dots,P_8$. Hence $d_i\le 1$ and when $d_i=1$ the line $\bar E'_i$ must pass through $P_7$ and $P_8$. 

This excludes the 6-tuple $(d_0,d_1,d_2,d_3,d_4,d_5)=(8,0,0,0,1,1)$ since both $\bar E'_4$ and $\bar E'_5$ would be lines through the points $P_7$ and $P_8$. 

Then we have the following list of 6-tuples $(d_0,d_1,d_2,d_3,d_4,d_5)$:
\begin{equation*}
(8,2,0,0,0,0), \quad (8,1,1,0,0,0), \quad (8,1,0,0,1,0).
\end{equation*}

For $i=1,2,3$, since $\bar E'_i\bar N_2=1$ (hence $\sum_{j=1}^8m_j=3d_i-1$), we have 
\begin{align*}
3&=\bar E'_i\bar B_0=8d_i-3(m_1+\dots+m_6)-2(m_7+m_8)\\
&=8d_i-3\sum_{j=1}^8 m_j+(m_7+m_8)=-d_i+3+m_7+m_8
\end{align*}
hence 
\begin{equation}\label{e.m123}
m_7+m_8=d_i\le 2.
\end{equation}

We now study the 6-tuple of degrees $(8,2,0,0,0,0)$.
Using \eqref{e.m45}, \eqref{e.m123} and the fact that $2\bar B_0+\bar E'$ has total multiplicity 6 at each of the points $P_1,\dots, P_8$ we find the following configuration
\begin{equation*}
\begin{tabular}{c|cccccccc}
  &$P_1$  &$P_2$ &$P_3$ &$P_4$ &$P_5$ &$P_6$ &$P_7$ &$P_8$\\
\hline
8 &3 &3 &3 &3 &3 &3 &2 &2\\
2 &1 &1 &1 &0 &0 &0 &1 &1\\
0 &-1 &0 &0 &0 &0 &0 &0 &0\\
0 &0 &0 &0 &0 &0 &0 &0 &0\\
0 &0 &-1 &0 &0 &0 &0 &1 &0\\
0 &0 &0 &-1 &0 &0 &0 &0 &1\\
\end{tabular}
\end{equation*}

For the 6-tuple $(8,1,1,0,0,0)$ we find
\begin{equation*}
\begin{tabular}{c|cccccccc}
  &$P_1$  &$P_2$ &$P_3$ &$P_4$ &$P_5$ &$P_6$ &$P_7$ &$P_8$\\
\hline
8 &3 &3 &3 &3 &3 &3 &2 &2\\
1 &1 &0 &0 &0 &0 &0 &1 &0\\
1 &0 &1 &0 &0 &0 &0 &0 &1\\
0 &0 &0 &0 &0 &0 &0 &0 &0\\
0 &-1 &0 &0 &0 &0 &0 &1 &0\\
0 &0 &-1 &0 &0 &0 &0 &0 &1\\
\end{tabular}
\end{equation*}

For the 6-tuple $(8,1,0,0,1,0)$ we have 
\begin{equation*}
\begin{tabular}{c|cccccccc}
  &$P_1$  &$P_2$ &$P_3$ &$P_4$ &$P_5$ &$P_6$ &$P_7$ &$P_8$\\
\hline
8 &3 &3 &3 &3 &3 &3 &2 &2\\
1 &1 &0 &0 &0 &0 &0 &1 &0\\
0 &-1 &0 &0 &0 &0 &0 &0 &0\\
0 &0 &0 &0 &0 &0 &0 &0 &0\\
1 &0 &1 &0 &0 &0 &0 &1 &1\\
0 &0 &-1 &0 &0 &0 &0 &0 &1\\
\end{tabular}
\end{equation*}

\begin{remark}
The conditions \eqref{e.m45} and \eqref{e.m123}, the total multiplicity 6 of $2\bar B_0+\bar E'$ at each of the points $P_1,\dots,P_8$ and the computation of the intersection numbers
\begin{equation*}
\bar E'_i\bar E'_j=\begin{cases}
0 &i=4,5,\ i\ne j\\
\ge 1 &i\ne j,\ 1\le i,j\le 3\\
-1 &1\le i=j \le 3
\end{cases}
\end{equation*}
are sufficient to uniquely determine  the configuration of points for each 6-tuple of degrees $(d_0,d_1,d_2,d_3,d_4,d_5)$.
\end{remark}

\begin{lemma}
The three above configurations are equivalent up to a finite number of quadratic transformations.	
\end{lemma}

\begin{proof}
We consider the 6-tuple $(8,1,1,0,0,0)$.
Let us  apply a quadratic transformation based at $P_1,P_2,P_8$. Since $P_1$ is a point of maximal multiplicity for both $\bar B_0$ and $\bar E'_1$ while  $P_2$ is a point of maximal multiplicity for both $\bar B_0$ and $\bar E'_2$, they cannot be infinitely near to any other point. Moreover $P_8$ is proximate to $P_2$ since the line $E'_2$ joins the two points and does not pass through any of the other points. Hence a quadratic transformation based at $P_1,P_2,P_8$ is well-defined and we obtain $(8,1,0,0,0,1)$.

We now show that $(8,1,1,0,0,0)$ is equivalent to $(8,2,0,0,0,0)$. We know that $P_8$ is proximate to $P_2$. A similar argument shows that $P_7$ is proximate to $P_1$. Let us now consider the points $P_3,P_4,P_5,P_6$. We claim that none of them can be proximate to $P_1$ or to $P_2$. If this was the case, in fact, the octic should satisfy the proximity inequalities  (see \cite{Ca})
$3\ge 3+2=5$
and we get a contradiction. Hence at least one of them, say $P_3$, has to be a planar point and we can perform a quadratic transformation based at $P_2,P_3,P_8$ obtaining the 6-tuple $(8,2,0,0,0,0)$.
\end{proof}

Thus all the 6-tuples are equivalent up to quadratic transformations and we can reduce to one of them, say $(8,2,0,0,0,0)$. We also note that  the curve $E'_3$ is contracted on $W$. In particular  we have 
$E'_3\sum_{j=1}^2Z'_j=2, E'_3Z''=-1$. 

Moreover since $E'_3\le Z''$ we find $E'_3N_2=0$ and then from the Index theorem   we have
$(E'_3+Z'_j)^2=-3-1+2E'_3Z'_j<0$ 
 (recall that $Z'_jN_2=0$ for any $j=1,2$)
hence $E'_3Z'_1=E'_3Z'_2=1$.

We now look at the surface $Y$ which is isomorphic to the plane blown up at 14 points. Let us denote by $P_9$ the point obtained by contracting $Z''$, $P_{10}$ and $P_{11}$ the points  obtained by contracting the cycles $Z'_j$, $P_{12}$ and $P_{13}$ the contractions of $Z_1$ and $Z_2$ and, finally, $P_{14}$ the contraction of $G'$. From section \ref{s.3l-1} 
\begin{equation*}
2B_0+E'\equiv 6G'+3\sum_{i=1}^nZ_i+2\sum_{i=1}^2Z'_i+Z''-6K_Y
\end{equation*}
hence the total multiplicity of $2B_0+E'$ at $P_9$ is 5, at $P_{10}$ and $P_{11}$ is 4, at $P_{12}$ and $P_{13}$ is 3 and it is 0 at $P_{14}$. 
Therefore we can write the following table 
\begin{equation*}
\begin{tabular}{c|cccccccccccccc}
  &$P_1$  &$P_2$ &$P_3$ &$P_4$ &$P_5$ &$P_6$ &$P_7$ &$P_8$ &$P_9$ &$P_{10}$ &$P_{11}$ &$P_{12}$ &$P_{13}$ &$P_{14}$\\
\hline
8 &3 &3 &3 &3 &3 &3 &2 &2 &2 &1 &1 &1 &1 &0\\
2 &1 &1 &1 &0 &0 &0 &1 &1 &1 &1 &0 &0 &0 &0\\
0 &-1 &0 &0 &0 &0 &0 &0 &0 &1 &0 &1 &0 &0 &0\\
0 &0 &0 &0 &0 &0 &0 &0 &0 &-1 &1 &1 &0 &0 &0\\
0 &0 &-1 &0 &0 &0 &0 &1 &0 &0 &0 &0 &1 &0 &0\\
0 &0 &0 &-1 &0 &0 &0 &0 &1 &0 &0 &0 &0 &1 &0\\
\end{tabular}
\end{equation*}

Let us now see what the images of $F'$ and $H'$ are. We know that $F'G'=H'G'=1$ hence their images pass through the point $P_{14}$. We also know they have no intersection with $B_0$ and with any of the curves $E'_i$.

Let us now consider $F'$. The computation for $H'$ is similar. Its plane image is a curve of degree $d$ with multiplicities $m_1,\dots,m_{14}$ at the points $P_1,\dots, P_{14}$. From the above remarks we find the following relations
\begin{equation}\label{e.f}
\begin{cases}
m_{14}=1\\
3\sum_{i=1}^6m_i+2\sum_{i=7}^9m_i+\sum_{i=10}^{13}m_i=8d\\
\sum_{i=1}^3m_i+\sum_{i=7}^{10}m_i=2d\\
-m_1+m_9+m_{11}=0\\
-m_9+m_{10}+m_{11}=0\\
-m_2+m_7+m_{12}=0\\
-m_3+m_8+m_{13}=0\\
\end{cases}
\end{equation}

One can easily see that  $F'$ (hence $H'$) is not contracted on $W$. Since $F'N_1=F'N_2=0$ we have $F'\sum_{i=1}^2Z_i=F'\sum_{i=1}^2Z'_i=0=F'Z''$. This forces $F'Z_i=F'Z'_i=0$ for $i=1,2$, hence 
$m_i=0$ for $9\le i\le 13$.
We can then rewrite \eqref{e.f} as
\begin{equation}\label{e.f2}
\begin{cases}
m_{14}=1\\
m_1=0\\
m_2+m_3=d\\
m_7=m_2\\
m_8=m_3\\
m_4+m_5+m_6=d\\
\end{cases}
\end{equation}

Since $F'$ is a $(-3)$-curve from \eqref{e.f2} we have
\begin{equation*}
-3=d^2-\sum_{i=1}^{14}m_i^2=-(d-2m_2)^2-m_4^2-m_5^2-m_6^2-1
\end{equation*}
hence
$(d-2m_2)^2+m_4^2+m_5^2+m_6^2=2$.
First of all we note that
\begin{equation*}
2\ge m_4^2+m_5^2+m_6^2\ge m_4+m_5+m_6=d
\end{equation*}
forces $d\le 2$. If $d=2$ we find $m_2=1$ and $(m_4,m_5,m_6)=(1,1,0),(1,0,1)$ or $(0,1,1)$. 
Using \eqref{e.f2} we find $m_2=m_3=m_7=m_8=m_{14}=1$. Hence $F'$ and $H'$ cannot be both sent to conics, since otherwise they should have at least 5 common points while $F'H'=0$ on $Y$.

When $d=1$ we find $-1\le 1-2m_2\le 1$ hence either $m_2=m_7=0,m_3=m_8=1$ or $m_2=m_7=1,m_3=m_8=0$.
 $F'$ and $H'$ cannot be sent to a conic and a line respectively, since they should have at least 3 common points ($P_2,P_7,P_{14}$ or $P_3,P_8,P_{14}$). Contradiction.

If $d=0$ then $-1\le -2m_2\le 1$ forces $m_2=0$. Hence from \eqref{e.f2} we get $m_2=m_3=m_7=m_8=0, m_4+m_5+m_6=0$. Thus
$\{m_4,m_5,m_6\}=\{1,0,-1\}$.

From the above analysis either $F'$ and $H'$ are both sent to lines or one of them is contracted on $\mathbb P^2$. In the former case we have the configuration
\begin{equation*}
\begin{tabular}{c|cccccccccccccc}
  &$P_1$  &$P_2$ &$P_3$ &$P_4$ &$P_5$ &$P_6$ &$P_7$ &$P_8$ &$P_9$ &$P_{10}$ &$P_{11}$ &$P_{12}$ &$P_{13}$ &$P_{14}$\\
\hline
8 &3 &3 &3 &3 &3 &3 &2 &2 &2 &1 &1 &1 &1 &0\\
2 &1 &1 &1 &0 &0 &0 &1 &1 &1 &1 &0 &0 &0 &0\\
0 &-1 &0 &0 &0 &0 &0 &0 &0 &1 &0 &1 &0 &0 &0\\
0 &0 &0 &0 &0 &0 &0 &0 &0 &-1 &1 &1 &0 &0 &0\\
0 &0 &-1 &0 &0 &0 &0 &1 &0 &0 &0 &0 &1 &0 &0\\
0 &0 &0 &-1 &0 &0 &0 &0 &1 &0 &0 &0 &0 &1 &0\\
1 &0 &1 &0 &1 &0 &0 &1 &0 &0 &0 &0 &0 &0 &1\\
1 &0 &0 &1 &0 &1 &0 &0 &1 &0 &0 &0 &0 &0 &1\\
\end{tabular}
\end{equation*}

In the latter case we can assume that the contracted curve (resp. one of the contracted curves) has $m_4=1,m_5=-1,m_6=0$. If the second curve is a conic we find the configuration
\begin{equation*}
\begin{tabular}{c|cccccccccccccc}
  &$P_1$  &$P_2$ &$P_3$ &$P_4$ &$P_5$ &$P_6$ &$P_7$ &$P_8$ &$P_9$ &$P_{10}$ &$P_{11}$ &$P_{12}$ &$P_{13}$ &$P_{14}$\\
\hline
8 &3 &3 &3 &3 &3 &3 &2 &2 &2 &1 &1 &1 &1 &0\\
2 &1 &1 &1 &0 &0 &0 &1 &1 &1 &1 &0 &0 &0 &0\\
0 &-1 &0 &0 &0 &0 &0 &0 &0 &1 &0 &1 &0 &0 &0\\
0 &0 &0 &0 &0 &0 &0 &0 &0 &-1 &1 &1 &0 &0 &0\\
0 &0 &-1 &0 &0 &0 &0 &1 &0 &0 &0 &0 &1 &0 &0\\
0 &0 &0 &-1 &0 &0 &0 &0 &1 &0 &0 &0 &0 &1 &0\\
0 &0 &0 &0 &1 &-1 &0 &0 &0  &0 &0 &0 &0 &0 &1\\
2 &0 &1 &1 &0 &1 &1 &1 &1 &0 &0 &0 &0 &0 &1\\
\end{tabular}
\end{equation*}

If it is a line we find
\begin{equation*}
\begin{tabular}{c|cccccccccccccc}
  &$P_1$  &$P_2$ &$P_3$ &$P_4$ &$P_5$ &$P_6$ &$P_7$ &$P_8$ &$P_9$ &$P_{10}$ &$P_{11}$ &$P_{12}$ &$P_{13}$ &$P_{14}$\\
\hline
8 &3 &3 &3 &3 &3 &3 &2 &2 &2 &1 &1 &1 &1 &0\\
2 &1 &1 &1 &0 &0 &0 &1 &1 &1 &1 &0 &0 &0 &0\\
0 &-1 &0 &0 &0 &0 &0 &0 &0 &1 &0 &1 &0 &0 &0\\
0 &0 &0 &0 &0 &0 &0 &0 &0 &-1 &1 &1 &0 &0 &0\\
0 &0 &-1 &0 &0 &0 &0 &1 &0 &0 &0 &0 &1 &0 &0\\
0 &0 &0 &-1 &0 &0 &0 &0 &1 &0 &0 &0 &0 &1 &0\\
0 &0 &0 &0 &1 &-1 &0 &0 &0 &0 &0 &0 &0 &0 &1\\
1 &0 &0 &1 &0 &1 &0 &0 &1 &0 &0 &0 &0 &0 &1\\
\end{tabular}
\end{equation*}
while if they are both contracted we have
\begin{equation*}
\begin{tabular}{c|cccccccccccccc}
  &$P_1$  &$P_2$ &$P_3$ &$P_4$ &$P_5$ &$P_6$ &$P_7$ &$P_8$ &$P_9$ &$P_{10}$ &$P_{11}$ &$P_{12}$ &$P_{13}$ &$P_{14}$\\
\hline
8 &3 &3 &3 &3 &3 &3 &2 &2 &2 &1 &1 &1 &1 &0\\
2 &1 &1 &1 &0 &0 &0 &1 &1 &1 &1 &0 &0 &0 &0\\
0 &-1 &0 &0 &0 &0 &0 &0 &0 &1 &0 &1 &0 &0 &0\\
0 &0 &0 &0 &0 &0 &0 &0 &0 &-1 &1 &1 &0 &0 &0\\
0 &0 &-1 &0 &0 &0 &0 &1 &0 &0 &0 &0 &1 &0 &0\\
0 &0 &0 &-1 &0 &0 &0 &0 &1 &0 &0 &0 &0 &1 &0\\
0 &0 &0 &0 &1 &-1 &0 &0 &0  &0 &0 &0 &0 &0 &1\\
0 &0 &0 &0 &0 &1 &-1 &0 &0  &0 &0 &0 &0 &0 &1\\
\end{tabular}
\end{equation*}

\begin{lemma}
The  configurations given by the 6-tuples $(8,2,0,0,0,0,1,1)$ and $(8,2,0,0,0,0,0,1)$ are equivalent up to quadratic transformations.
\end{lemma}

\begin{proof}
Let us study the configuration with two lines. One can easily see that $P_2$ and $P_3$ are planar points since they are of maximal multiplicity for the octic, the conic and one of the two lines simultaneously. Moreover $P_7$ is proximate to $P_2$ while $P_3$ is proximate to $P_8$. Thus $P_1$ cannot be proximate to $P_2$ or to $P_3$ since otherwise the octic would contradict the proximity inequalities  (see \cite{Ca})
$3\ge 3+2=5$.

With a similar argument one can show that $P_4$ and $P_5$ are planar points too.
If we base a quadratic transformation at $P_3,P_4,P_8$ we obtain the configuration with a line and a contracted curve.
\end{proof}

We look at the eigenvalues of the curves $B_0$, $E'_i, 1\le i\le 5$, $F'$ and $H'$ for the action of the automorphism of order 3. We know that $F'$ and $H'$ correspond to different eigenvalues since they come from the blow-up of a singularity of type $A_2$ (see \cite{Ca, T}). From now on let us set $\omega:=e^{2\pi i/3}$. If $B_0$ corresponds to the eigenvalue $\omega$ then it appears with multiplicity 1 in the branch locus of the simple triple cover associated to $X\longrightarrow Y=X/(\mathbb Z/3\mathbb Z)$. Let us assume that $E'_i$ corresponds to the eigenvalue $\omega^{\nu_i}$, $F'$ corresponds to the eigenvalue $\omega^{\nu_F}$ and $H'$ to $\omega^{2\nu_F}$. 

\begin{prop}\label{p.3l-1}
The case $n=3\ell-1,n'=2,n''=1$ cannot occur with multi-degrees $(d_0,d_1,d_2,d_3,d_4,d_5,d_F,d_H)=(8,2,0,0,0,0,1,1)$ and $(8,2,0,0,0,0,1)$.
\end{prop}

\begin{proof}
We have already shown that the two configurations in the statement 
are equivalent up to quadratic transformations.
Let us consider 
\begin{equation*}
\begin{tabular}{c|cccccccccccccc}
  &$P_1$  &$P_2$ &$P_3$ &$P_4$ &$P_5$ &$P_6$ &$P_7$ &$P_8$ &$P_9$ &$P_{10}$ &$P_{11}$ &$P_{12}$ &$P_{13}$ &$P_{14}$\\
\hline
8 &3 &3 &3 &3 &3 &3 &2 &2 &2 &1 &1 &1 &1 &0\\
2 &1 &1 &1 &0 &0 &0 &1 &1 &1 &1 &0 &0 &0 &0\\
0 &-1 &0 &0 &0 &0 &0 &0 &0 &1 &0 &1 &0 &0 &0\\
0 &0 &0 &0 &0 &0 &0 &0 &0 &-1 &1 &1 &0 &0 &0\\
0 &0 &-1 &0 &0 &0 &0 &1 &0 &0 &0 &0 &1 &0 &0\\
0 &0 &0 &-1 &0 &0 &0 &0 &1 &0 &0 &0 &0 &1 &0\\
1 &0 &1 &0 &1 &0 &0 &1 &0 &0 &0 &0 &0 &0 &1\\
1 &0 &0 &1 &0 &1 &0 &0 &1 &0 &0 &0 &0 &0 &1\\
\end{tabular}
\end{equation*}

Since the total degree of the branch curve on $\mathbb P^2$ has to be a  multiple of 3 and since the two lines correspond to different eigenvalues, the conic appears with multiplicity 2 in the branch divisor, hence $\nu_1\equiv 2\mod 3$.

The points $P_2$, $P_{3}$ are not infinitely near to any other point since they are the only points which are triple for the octic and simple for both the conic and one of the two lines. 
The total multiplicity at $P_3$ of the branch divisor has to be a multiple of 3. Then we obtain the  equation 
\begin{equation*}
3+\nu_1+\nu_5+\nu_H\equiv 3+2+\nu_5+\nu_H \equiv 0\mod 3
\end{equation*}
which forces $\nu_H+\nu_5\equiv 2\nu_F+\nu_5\equiv 1 \mod 3$.

On the other hand the same computation for $P_2$ gives us
\begin{equation*}
3+\nu_1+\nu_4+\nu_F\equiv 3+2+\nu_4+\nu_F \equiv 0\mod 3
\end{equation*}
which forces $\nu_F+\nu_4\equiv 1 \mod 3$.
Then since $\nu_i,\nu_F\equiv 1,2 \mod 3$ we find
$\nu_F\equiv \nu_4\equiv 2\mod 3$
hence $\nu_5\equiv 0\mod 3$. Contradiction.
\end{proof}

\begin{prop}\label{p.3l-12}
The case $n=3\ell-1,n'=2,n''=1$ cannot occur with degrees $(d_0,d_1,d_2,d_3,d_4,d_5,d_F,d_H)=(8,2,0,0,0,0,0,2)$. 
\end{prop}

\begin{proof}
Let us consider the points $P_1,P_2,P_3$. They are of maximal multiplicity for both the octic and one of the two conics hence they cannot be infnitely near to any of the points $P_j, j\ge 4$. Since there is an irreducible  conic passing through all the three points, we can perform a quadratic transformation based at $P_1,P_2,P_3$ and we obtain the following configuration
\begin{equation*}
\begin{tabular}{c|cccccccccccccc}
  &$P_1$  &$P_2$ &$P_3$ &$P_4$ &$P_5$ &$P_6$ &$P_7$ &$P_8$ &$P_9$ &$P_{10}$ &$P_{11}$ &$P_{12}$ &$P_{13}$ &$P_{14}$\\
\hline
7 &2 &2 &2 &3 &3 &3 &2 &2 &2 &1 &1 &1 &1 &0\\
1 &0 &0 &0 &0 &0 &0 &1 &1 &1 &1 &0 &0 &0 &0\\
1 &0 &1 &1 &0 &0 &0 &0 &0 &1 &0 &1 &0 &0 &0\\
0 &0 &0 &0 &0 &0 &0 &0 &0 &-1 &1 &1 &0 &0 &0\\
1 &1 &0 &1 &0 &0 &0 &1 &0 &0 &0 &0 &1 &0 &0\\
1 &1 &1 &0 &0 &0 &0 &0 &1 &0 &0 &0 &0 &1 &0\\
0 &0 &0 &0 &1 &-1 &0 &0 &0  &0 &0 &0 &0 &0 &1\\
2 &0 &1 &1 &0 &1 &1 &1 &1 &0 &0 &0 &0 &0 &1\\
\end{tabular}
\end{equation*}

We now show that this new configuration cannot occur.

Let us consider the points $P_4,P_5,P_6$. Since they are triple points for the septic they cannot be infinitely near to any other point $P_j$, $j\le 3$ or $j\ge 7$. Moreover the conic $H'$ passes through $P_5$ and $P_6$ but not through $P_4$. Hence one among $P_5$ and $P_6$ has to be a planar point. 

If $P_6$ was planar, then the total multiplicity of $P_6$ in the branch divisor of the simple triple cover has to be a multiple of 3. Thus
$3+\nu_H\equiv 0 \mod 3$ 
which forces $\nu_H\equiv 0 \mod 3$. We get a contradiction since $H'$ is an irreducible component of the branch divisor.
Thus $P_5$ is a planar point and $P_6$ is proximate to $P_5$. But then when we blow up $P_5$ the exceptional divisor $F'$ should pass through $P_6$. Contradiction. 
\end{proof}

\begin{prop}\label{p.3l-13}
The case $n=3\ell-1,n'=2,n''=1$ cannot occur with degrees $(d_0,d_1,d_2,d_3,d_4,d_5,d_F,d_H)=(8,2,0,0,0,0,0,0)$. 
\end{prop}

\begin{proof}
Let us consider the points $P_1,P_2,P_3$. They are of maximal multiplicity for both the octic and one of the two conics hence they cannot be infnitely near to any of the points $P_j,j\ge 4$. Since there is an irreducible  conic passing through all the three points, we can perform a quadratic transformation based at $P_1,P_2,P_3$ and we obtain the following configuration
\begin{equation*}
\begin{tabular}{c|cccccccccccccc}
  &$P_1$  &$P_2$ &$P_3$ &$P_4$ &$P_5$ &$P_6$ &$P_7$ &$P_8$ &$P_9$ &$P_{10}$ &$P_{11}$ &$P_{12}$ &$P_{13}$ &$P_{14}$\\
\hline
7 &2 &2 &2 &3 &3 &3 &2 &2 &2 &1 &1 &1 &1 &0\\
1 &0 &0 &0 &0 &0 &0 &1 &1 &1 &1 &0 &0 &0 &0\\
1 &0 &1 &1 &0 &0 &0 &0 &0 &1 &0 &1 &0 &0 &0\\
0 &0 &0 &0 &0 &0 &0 &0 &0 &-1 &1 &1 &0 &0 &0\\
1 &1 &0 &1 &0 &0 &0 &1 &0 &0 &0 &0 &1 &0 &0\\
1 &1 &1 &0 &0 &0 &0 &0 &1 &0 &0 &0 &0 &1 &0\\
0 &0 &0 &0 &1 &-1 &0 &0 &0  &0 &0 &0 &0 &0 &1\\
0 &0 &0 &0 &0 &1 &-1 &0 &0 &0 &0 &0 &0 &0 &1\\
\end{tabular}
\end{equation*}

We now show that this new configuration cannot occur.
Let us consider the points $P_4,P_5,P_6$. Since they are triple points for the septic they cannot be infinitely near to any other point $P_j$, $j\le 3$ or $j\ge 7$. Moreover  $P_4$ is proximate to $P_5$ which is also proximate to $P_6$. In particular $P_6$ is a planar point.

Since  $P_6$ is planar, the total multiplicity of $P_6$ in the branch divisor of the simple triple cover has to be a multiple of 3. Thus
$3+\nu_H\equiv 0 \mod 3$ 
which forces $\nu_H\equiv 0 \mod 3$. We get a contradiction since $H'$ is an irreducible component of the branch divisor.
\end{proof}

Hence we obtain
\begin{theorem}\label{t.3l-1}
The case  $n=3\ell-1,n'=2,n''=1$ cannot occur.
\end{theorem}

With analogous computations we can prove (cf. \cite[sections 5.2.1, 5.2.2]{Pa})
\begin{theorem}\label{t.no3lDP}
The cases $n=3\ell,n'=0,n''=n'''=1$ with $\ell=0$,  $n=3\ell-2,n'=5$ and  $n=3l-3$ cannot occur.
\end{theorem}

Collecting the proofs of theorems \ref{t.i}, \ref{t.ii}, \ref{t.iii2},  \ref{t.no0}, \ref{t.no1rul}, \ref{t.no1},  \ref{t.no3lDP1},  \ref{t.3l-1} and \ref{t.no3lDP} we eventually obtain theorem \ref{t.final}.

\end{document}